\documentclass[12pt]{amsart}
\usepackage{graphicx}
\usepackage{amsmath,amsfonts}
\usepackage{amsthm,amssymb,latexsym}
\usepackage[active]{srcltx}
\DeclareMathOperator{\rank}{rank}

\newtheorem{theorem}{Theorem}[section]
\newtheorem{corollary}[theorem]{Corollary}
\newtheorem{lemma}[theorem]{Lemma}
\newtheorem{fact}[theorem]{Fact}
\newtheorem*{fact*}{Fact}

\theoremstyle{definition}

\newtheorem{remark}[theorem]{Remark}
\numberwithin{equation}{section}

\newcommand{\N}{\mathbb N}
\newcommand{\Z}{\mathbb Z}

\newcommand{\A}{\mathbb A}
\newcommand{\F}{\mathbb F}
\newcommand{\K}{\mathbb K}

\newcommand{\Pp}{\mathbb P}

\newcommand{\bfs}{\boldsymbol}

\newcommand{\fq}{\F_{\hskip-0.7mm q}}

\newcommand{\cfq}{\overline{\F}_{\hskip-0.7mm q}}
\def\ifm#1#2{\relax \ifmmode#1\else#2\fi}

\newcommand{\klk}    {\ifm {,\ldots,} {$,\ldots,$}}

\newcommand{\plp}    {\ifm {+\cdots+} {$+\ldots+$}}

\makeatletter
\@namedef{subjclassname@2010}{%
  \textup{2010} Mathematics Subject Classification}
\makeatother

\begin{document}

\title[Value set of small families II]{On the
value set of small families of polynomials over a finite field, II}
\author[G. Matera et al.]{
%
%
Guillermo Matera${}^{1,2}$,
%
Mariana P\'erez${}^1$,
%
and Melina Privitelli${}^{2,3}$}

\address{${}^{1}$Instituto del Desarrollo Humano,
Universidad Nacional de Gene\-ral Sarmiento, J.M. Guti\'errez 1150
(B1613GSX) Los Polvorines, Buenos Aires, Argentina}
\email{\{gmatera,\,vperez\}@ungs.edu.ar}
\address{${}^{2}$ National Council of Science and Technology (CONICET),
Ar\-gentina}\email{mprivitelli@conicet.gov.ar}
\address{${}^{3}$Instituto de Ciencias,
Universidad Nacional de Gene\-ral Sarmiento, J.M. Guti\'errez 1150
(B1613GSX) Los Polvorines, Buenos Aires, Argentina}

\thanks{The authors were partially supported by the grants
PIP 11220090100421 CONICET and STIC--AmSud ``Dynalco''.}%

\date{\today}
\begin{abstract}
We obtain an estimate on the average cardinality of the value set of
any family of monic polynomials of $\fq[T]$ of degree $d$ for which
$s$ consecutive coefficients $a_{d-1}\klk a_{d-s}$ are fixed. Our
estimate asserts that
$\mathcal{V}(d,s,\bfs{a})=\mu_d\,q+\mathcal{O}(q^{1/2})$, where
$\mathcal{V}(d,s,\bfs{a})$ is such an average cardinality,
$\mu_d:=\sum_{r=1}^d{(-1)^{r-1}}/{r!}$ and $\bfs{a}:=(a_{d-1}\klk
a_{d-s})$. We also prove that
$\mathcal{V}_2(d,s,\bfs{a})=\mu_d^2\,q^2+\mathcal{O}(q^{3/2})$,
where that $\mathcal{V}_2(d,s,\bfs{a})$ is the average second moment
on any family of monic polynomials of $\fq[T]$ of degree $d$ with
$s$ consecutive coefficients fixed as above. Finally, we show that
$\mathcal{V}_2(d,0)=\mu_d^2\,q^2+\mathcal{O}(q)$, where
$\mathcal{V}_2(d,0)$ denotes the average second moment of all monic
polynomials in $\fq[T]$ of degree $d$ with $f(0)=0$. All our
estimates hold for fields of characteristic $p>2$ and provide
explicit upper bounds for the constants underlying the
$\mathcal{O}$--notation in terms of $d$ and $s$ with ``good''
behavior. Our approach reduces the questions to estimate the number
of $\fq$--rational points with pairwise--distinct coordinates of a
certain family of complete intersections defined over $\mathbb F_q$.
A critical point for our results is an analysis of the singular
locus of the varieties under consideration, which allows to obtain
rather precise estimates on the corresponding number of
$\fq$--rational points.
\end{abstract}



\maketitle
%
%
\section{Introduction}
Let $\fq$ be the finite field of $q$ elements, 
let $T$ be an indeterminate over $\fq$ and let $f\in\fq[T]$. We
define the value set $\mathcal{V}(f)$ of $f$ as
$\mathcal{V}(f):=|\{f(c):c\in\fq\}|$ (cf. \cite{LiNi83}). This paper
is a continuation of \cite{CeMaPePr13} and is concerned with results
on the average value set of certain families of polynomials of
$\fq[T]$.

Let $\mathcal{V}(d,0)$ denote the average value set $\mathcal{V}(f)$
when $f$ ranges over all monic polynomials in $\fq[T]$ of degree $d$
with $f(0)=0$. Then it is well-known that
\begin{equation}\label{eq: V(d,0) Cohen}
\mathcal{V}(d,0)=\sum_{r=1}^d(-1)^{r-1}\binom{q}{r}q^{1-r}=\mu_d\,q
+\mathcal{O}(1),\end{equation}
where $\mu_d:=\sum_{r=1}^d{(-1)^{r-1}}/{r!}$ and the constant
underlying the $\mathcal{O}$--notation depends only on $d$ (see
\cite{Uchiyama55a}, \cite{Cohen73}).

On the other hand, if some of the coefficients of $f$ are fixed, the
results on the average value of $\mathcal{V}(f)$ are less precise.
More precisely, let be given $s$ with $1\le s\le d-2$ and
$\bfs{a}:=(a_{d-1}\klk a_{d-s})\in\fq^s$. For every
$\boldsymbol{b}:=(b_{d-s-1}\klk b_1)$, let
$$f_{\boldsymbol{b}}:=f_{\boldsymbol{b}}^{\bfs{a}}
:=T^d+\sum_{i=1}^sa_{d-i}T^{d-i}+\sum_{i=s+1}^{d-1}b_{d-i}T^{d-i}.$$
Then for $p:=\mathrm{char}(\fq)>d$,
\begin{equation}\label{eq: average value set - cohen}
\mathcal{V}(d,s,\bfs{a}):=
\frac{1}{q^{d-s-1}}\sum_{\boldsymbol{b}\in\fq^{d-s-1}}\mathcal{V}(f_{\boldsymbol{b}})=
\mu_d\,q+\mathcal{O}(q^{1/2}),
\end{equation}
where the constant underlying the $\mathcal{O}$--notation depends
only on $d$ and $s$ (see \cite{Uchiyama55b}, \cite{Cohen72}). In a
previous paper \cite{CeMaPePr13} we obtain the following explicit
estimate for $q>d$ and $1\le s\le \frac{d}{2}-1$:
\begin{equation}\label{eq: average value set - CMPP}
|\mathcal{V}(d,s,\bfs{a})-\mu_d\,q|\le \frac{e^{-1}}{2}+
\frac{(d-2)^5e^{2\sqrt{d}}}{2^{d-2}} +\frac{7}{q}.
\end{equation}
This result holds without any restriction on the characteristic $p$
of $\fq$ and shows that $\mathcal{V}(d,s,\bfs{a})=\mu_d\,
q+\mathcal{O}(1)$. On the other hand, it must be said that (\ref{eq:
average value set - CMPP}) is valid for $1\le s\le d/2-1$, while
(\ref{eq: average value set - cohen}) holds for a larger set of
values of $s$, namely for $1\le s\le d-2$.

In this paper we obtain an explicit estimate for
$\mathcal{V}(d,s,\bfs{a})$ which can be seen as a complement of
(\ref{eq: average value set - CMPP}). More precisely, we have the
following result.
\begin{theorem}\label{th: mean intro}
Let $q>d$ and let be given $s$ with $1 \leq s \leq d-4$ for $p>3$,
and $1\leq s \leq d-6$ for $p=3$. Then we have
\begin{equation}\label{eq: estimate para v intro}
\left|\mathcal{V}(d,s,\boldsymbol{a})-\mu_d\,q\right|\le d^2
2^{d-1}q^{1/2} +49\,d^{d+5} e^{2 \sqrt{d}-d}.
\end{equation}
\end{theorem}
We observe that (\ref{eq: estimate para v intro}) holds for a larger
set of values of $s$ than (\ref{eq: average value set - CMPP}),
although it does not holds for fields of characteristic 2. It might
also be worthwhile to remark that the estimate for
$|\mathcal{V}(f_{\boldsymbol{b}}) -\mu_d\,q|$ in (\ref{eq: estimate
para v intro}) does not behave as well as that of (\ref{eq: average
value set - CMPP}). On the other hand, it strengthens (\ref{eq:
average value set - cohen}) in that it provides an explicit estimate
for $|\mathcal{V}(f_{\boldsymbol{b}}) -\mu_d\,q|$ which holds for
fields of characteristic greater than 2.

A second aim for this paper is to provide estimates on the second
moment of the value set of the families of polynomials under
consideration. In connection with this matter, in \cite{Uchiyama56}
it is shown that, under the Riemann hypothesis for $L$-functions,
for $p>d$ we have
\begin{equation}\label{eq: V_2(d,0) uchiyama}
\mathcal{V}_2(d,0):=\frac{1}{q^{d-1}}\sum\mathcal{V}(f)^2=
\mu_d^2\,q^2+\mathcal{O}(q),
\end{equation}
where the sum ranges over all monic polynomials $f\in\fq[T]$ of
degree $d$ with $f(0)=0$ (see also \cite{KnKn90b} for results for
$d\ge q$). We obtain the following explicit version of (\ref{eq:
V_2(d,0) uchiyama}), which also holds for fields $\fq$ of small
characteristic.
\begin{theorem}\label{th: variance intro s=0}
Let $q>d$. If $d\ge 5$ for $p> 3$ and $d\ge 9$ for $p=3$, then
%
$$|\mathcal{V}_2(d,0)-\mu_d^2\,q^2|\le \big(d^22^{2d-2}+14^3
d^{2d+8}e^{4\sqrt{d}-2d}\big)\,q.$$
\end{theorem}

Our second result regarding second moments is an estimate on the
average second moment of the set of monic polynomials of degree $d$
with $s$ coefficients fixed. We obtain the following result.
\begin{theorem}\label{th: variance intro}
Let $q>d$ and let be given $s$ with $1 \leq s \leq d-4$ for $p> 3$,
and $1\leq s \leq d-6$ for $p=3$. Let
$\mathcal{V}_2(d,s,\boldsymbol{a}):=q^{-d+s+1} \sum_{\boldsymbol{b}
\in \mathbb{F}_q^{d-s-1}} \mathcal{V}(f_{\boldsymbol{b}})^2$. Then
we have
%
$$|\mathcal{V}_2(d,s,\boldsymbol{a})-\mu_d^2\,q^2|\le
d^22^{2d+1}q^{3/2}+14^3 d^{2d+6}e^{4 \sqrt{d}-2d}q.$$
\end{theorem}

Our approach to prove Theorem \ref{th: mean intro} shares certain
similarities with that of \cite{CeMaPePr13}. Indeed, we express the
quantity $\mathcal{V}(d,s,\bfs{a})$ in terms of the number
$\chi_r^{\bfs{a}}$ of certain ``interpolating sets'' with $d-s+1\le
r\le d$. More precisely, for $f_{\bfs{a}}:=T^d+a_{d-1}T^{d-1}\plp
a_{d-s}T^{d-s}$, we define $\chi_r^{\bfs{a}}$ as the number of
$r$--element subsets of $\fq$ at which $f_{\bfs{a}}$ can be
interpolated by a polynomial of degree at most $d-s-1$. In Section
\ref{sec: geometric approach mean} we show that the number
$\chi_r^{\bfs{a}}$ agrees with the number of $\fq$--rational
solutions with pairwise--distinct coordinates of a given
$\fq$--definable affine variety $\Gamma_r^*$ of $\cfq{\!}^{d-s+r}$
for $d-s+1\le r\le d$. In Section \ref{sec: geometry of Gamma_r
estrella} we establish a number of geometric properties of
$\Gamma_r^*$. This allows us to obtain, in Section \ref{sec: number
of points Gamma r}, a suitable estimate on the quantities
$\chi_r^{\bfs{a}}$ for $d-s+1\le r\le d$, and thus on
$\mathcal{V}(d,s,\bfs{a})$.

The proof of Theorems \ref{th: variance intro s=0} and \ref{th:
variance intro} follow a similar scheme to that of Theorem \ref{th:
mean intro}. We provide a detailed proof of Theorem \ref{th:
variance intro} in Sections \ref{sec: combinatorial preliminaries
variance}, \ref{sec: geometric approach variance}, \ref{sec:
geometry of Gamma_mn estrella} and \ref{sec: behavior V2} and a
sketch of the proof of Theorem \ref{th: variance intro s=0} in
Section \ref{sec: V_2(d,0)}. In Section \ref{sec: combinatorial
preliminaries variance} we obtain a combinatorial result which
expresses $\mathcal{V}_2 (d,s,\boldsymbol{a})$ in terms of the
number $\mathcal{S}_{m,n}^{\bfs{a}}$ of certain ``interpolatings
sets'' with $d-s+1\le m+n\le 2d$. In Section \ref{sec: geometric
approach variance} the number $\mathcal{S}_{m,n}^{\bfs{a}}$ is
expressed as the number of $\fq$--rational points with
pairwise--distinct coordinates of a given $\fq$--definable affine
variety $\Gamma_{m,n}^*$ of $\cfq{\!}^{d-s+1+ m+n}$ for each $m,n$
as above. In Section \ref{sec: geometry of Gamma_mn estrella} we
show certain results concerning the geometry of $\Gamma_{m,n}^*$,
which allow us to determine in Section \ref{sec: behavior V2} the
asymptotic behavior of $\mathcal{V}_2(d,s,\bfs{a})$. Finally, in
Section \ref{sec: V_2(d,0)} we discuss how the arguments of the
previous sections can be adapted in order to obtain a proof of
Theorem \ref{th: variance intro s=0}.

Finally, we remark that the analysis of the singular locus of the
varieties underlying the proofs of Theorems \ref{th: mean intro} and
\ref{th: variance intro} requires the study of discriminant locus of
the family of polynomials under consideration, namely the union of
the zero locus of the discriminants of all these polynomials. Such a
discriminant locus has been considered in \cite{FrSm84}, where it is
shown that it is absolutely irreducible for fields of characteristic
large enough. In an appendix we show that the discriminant locus is
absolutely irreducible for  fields of characteristic at least 3,
extending thus the main result of \cite{FrSm84}.
%
%
\section{Notions and notations}
Since our approach relies on tools of algebraic geometry, we briefly
collect the basic definitions and facts that we need in the sequel.
We use standard notions and notations which can be found in, e.g.,
\cite{Kunz85}, \cite{Shafarevich94}.

We denote by $\A^n$ the {\sf affine $n$--dimensional space}
$\cfq{\!}^{n}$ and by $\Pp^n$ the {\sf projective $n$--dimensional
space} over $\cfq{\!}^{n+1}$. Both spaces are endowed with their
respective {\sf Zariski topologies}, for which a closed set is the
zero locus of polynomials of $\cfq[X_1,\ldots, X_{n}]$ or of
homogeneous polynomials of  $\cfq[X_0,\ldots, X_{n}]$. For $\K:=\fq$
or $\K:=\cfq$, we say that a subset $V\subset \A^n$ is an {\sf
affine $\K$--variety} if it is the set of common zeros in $\A^n$ of
polynomials $F_1,\ldots, F_{m} \in \K[X_1,\ldots, X_{n}]$.
Correspondingly, a {\sf projective $\K$--variety} is the set of
common zeros in $\Pp^n$ of a family of homogeneous polynomials
$F_1,\ldots, F_m \in\K[X_0 ,\ldots, X_n]$. We shall denote by
$V(F_1\klk F_m)$ or $\{F_1=0,\dots,F_s=0\}$ the affine or projective
$\K$--variety consisting of the common zeros of polynomials $F_1\klk
F_m$. The set $V(\fq):=V\cap \fq^n$ is the set of {\sf $q$--rational
points} of $V$.

A $\K$--variety $V$ is $\K$--{\sf irreducible} if it cannot be
expressed as a finite union of proper $\K$--subvarieties of $V$.
Further, $V$ is {\sf absolutely irreducible} if it is irreducible as
a $\cfq$--variety. Any $\K$--variety $V$ can be expressed as an
irredundant union $V=\mathcal{C}_1\cup \cdots\cup\mathcal{C}_s$ of
irreducible (absolutely irreducible) $\K$--varieties, unique up to
reordering, which are called the {\sf irreducible} ({\sf absolutely
irreducible}) $\K$--{\sf components} of $V$.


For a $\K$-variety $V$ contained in $\A^n$ or $\Pp^n$, we denote by
$I(V)$ its {\sf defining ideal}, namely the set of polynomials of
$\K[X_1,\ldots, X_n]$, or of $\K[X_0,\ldots, X_n]$, vanishing on
$V$. The {\sf coordinate ring} $\K[V]$ of $V$ is defined as the
quotient ring $\K[X_1,\ldots,X_n]/I(V)$ or
$\K[X_0,\ldots,X_n]/I(V)$. The {\sf dimension} $\dim V$ of a
$\K$-variety $V$ is the length $r$ of the longest chain
$V_0\varsubsetneq V_1 \varsubsetneq\cdots \varsubsetneq V_r$ of
nonempty irreducible $\K$-varieties contained in $V$. A
$\K$--variety is called {\sf equidimensional} if all its irreducible
$\K$--components are of the same dimension.

The {\sf degree} $\deg V$ of an irreducible $\K$-variety $V$ is the
maximum number of points lying in the intersection of $V$ with a
linear space $L$ of codimension $\dim V$, for which $V\cap L$ is a
finite set. More generally, following \cite{Heintz83} (see also
\cite{Fulton84}), if $V=\mathcal{C}_1\cup\cdots\cup \mathcal{C}_s$
is the decomposition of $V$ into irreducible $\K$--components, we
define the degree of $V$ as
$$\deg V:=\sum_{i=1}^s\deg \mathcal{C}_i.$$
An important tool for our estimates is the following {\em B\'ezout
inequality} (see \cite{Heintz83}, \cite{Fulton84}, \cite{Vogel84}):
if $V$ and $W$ are $\K$--varieties, then the following inequality
holds:
\begin{equation}\label{eq: Bezout inequality}
\deg (V\cap W)\le \deg V \cdot \deg W.
\end{equation}

%

Elements $F_1 \klk F_{n-r}$ in $\K[X_1\klk X_n]$ or in $\K[X_0\klk
X_n]$ form a {\sf regular sequence} if $F_1$ is nonzero and each
$F_i$ is not a zero divisor in the quotient ring $\K[X_1\klk
X_n]/(F_1\klk F_{i-1})$ or $\K[X_0\klk X_n]/(F_1\klk F_{i-1})$ for
$2\le i\le n-r$. In such a case, the (affine or projective)
$\K$--variety $V:=V(F_1\klk F_{n-r})$ they define is equidimensional
of dimension $r$, and is called a {\sf set--theoretic} {\sf complete
intersection}. If the ideal $(F_1\klk F_{n-r})$ generated by
$F_1\klk F_{n-r}$ is radical, then we say that $V$ is an {\sf
ideal--theoretic} {\sf complete intersection}. If $V\subset\Pp^n$ is
an ideal--theoretic complete intersection defined over $\K$, of
dimension $r$ and degree $\delta$, and $F_1 \klk F_{n-r}$ is a
system of generators of $I(V)$, the degrees $d_1\klk d_{n-r}$ depend
only on $V$ and not on the system of generators. Arranging the $d_i$
in such a way that $d_1\geq d_2 \geq \cdots \geq d_{n-r}$, we call
$\boldsymbol{d}:=(d_1\klk d_{n-r})$ the {\sf multidegree} of $V$. In
particular, it follows that $\delta= \prod_{i=1}^{n-r}d_i$ holds.

Let $V$ be a variety contained in $\A^n$ and let $I(V)\subset
\cfq[X_1,\ldots, X_n]$ be the defining ideal of $V$. Let $\bfs{x}$
be a point of $V$. The {\sf dimension} $\dim_{\bfs{x}}V$ {\sf of}
$V$ {\sf at} $\bfs{x}$ is the maximum of the dimensions of the
irreducible components of $V$ that contain $\bfs{x}$. If
$I(V)=(F_1,\ldots, F_m)$, the {\sf tangent space}
$\mathcal{T}_{\bfs{x}}V$ to $V$ at $\bfs{x}$ is the kernel of the
Jacobian matrix $(\partial F_i/\partial X_j)_{1\le i\le m,1\le j\le
n}(\bfs{x})$ of the polynomials $F_1,\ldots, F_m$ with respect to
$X_1,\ldots, X_n$ at $\bfs{x}$. The point $\bfs{x}$ is {\sf regular}
if $\dim\mathcal{T}_{\bfs{x}}V=\dim_{\bfs{x}}V$ holds. Otherwise,
the point $\bfs{x}$ is called {\sf singular}. The set of singular
points of $V$ is the {\sf singular locus} $\mathrm{Sing}(V)$ of $V$.
A variety is called {\sf nonsingular} if its singular locus is
empty. For a projective variety, the concepts of tangent space,
regular and singular point can be defined by considering an affine
neighborhood of the point under consideration.

Let $V$ and $W$ be irreducibles $\K$--varieties of the same
dimension and let $f:V\to W$ be a regular map for which
$\overline{f(V)}=W$ holds, where $\overline{f(V)}$ denotes the
closure of $f(V)$ with respect to the Zariski $\K$--topology of $W$.
Then $f$ induces a ring extension $\K[W]\hookrightarrow \K[V]$ by
composition with $f$. We say that $f$ is a {\sf finite morphism} if
this extension is integral, namely if each element $\eta\in\K[V]$
satisfies a monic equation with coefficients in $\K[W]$. A basic
fact is that a finite morphism is necessarily closed. Another fact
concerning finite morphisms we shall use in the sequel is that the
preimage $f^{-1}(S)$ of an irreducible closed subset $S\subset W$ is
equidimensional of dimension $\dim S$.
%
%
\section{Estimating the mean $\mathcal{V}(d,s,\boldsymbol{a})$: a
geometric approach}\label{sec: geometric approach mean}
Let be given $s,d\in \N$ with $d<q$ and $1\le s\le d-2$ and
$\boldsymbol{a}:=(a_{d-1}\klk a_{d-s})\in\fq^s$. Denote
$f_{\bfs{a}}:=T^d+a_{d-1}T^{d-1}\plp a_{d-s}T^{d-s}$.
For every $\bfs{b}:=(b_{d-s-1}\klk b_1)\in\fq^{d-s-1}$, denote by
$f_{\bfs{b}}:=f_{\bfs{b}}^{\bfs{a}}\in\fq[T]$ the following
polynomial:
$$f_{\bfs{b}}:=f_{\bfs{a}}+
b_{d-s-1}T^{d-s-1}\plp b_1T.$$
Our first objective is to determine the asymptotic behavior of the
average value set
$$\mathcal{V}(d,s,\bfs{a}):=\frac{1}{q^{d-s-1}}
\sum_{\boldsymbol{b}\in\fq^{d-s-1}}\mathcal{V}(f_{\boldsymbol{b}}).$$
For this purpose, we have the following result.
\begin{theorem}[{\cite[Theorem 2.1]{CeMaPePr13}}]
\label{th: combinatorial reduction mean} With assumptions as above,
we have
\begin{equation}
  \mathcal{V}(d,s,\boldsymbol{a})= \sum_{r=1}^{d-s}(-1)^{r-1}
  \binom{q}{r}q^{1-r}
    +\frac{1}{q^{d-s-1}}\sum_{r=d-s+1}^{d}
   (-1)^{r-1}\chi_r^{\boldsymbol{a}},
\end{equation}
where $\chi_r ^{\boldsymbol{a}}$ denotes the number of subsets
$\chi_r$ of $\fq$ of $r$ elements such that there exists
$(\boldsymbol{b},b_0)\in \fq^{d-s}$  for which
$(f_{\boldsymbol{b}}+b_0)|_{\chi_r}\equiv 0$ holds.
\end{theorem}

According to this result, we have to determine the asymptotic
behavior of $\chi_r ^{\boldsymbol{a}}$ for $d-s+1\leq r\leq d$. In
\cite{CeMaPePr13} we introduce an affine $\fq$--variety
$V_r^{\bfs{a}}\subset\A^r$ in such a way that the number of
$q$--rational points of $V_r^{\bfs{a}}$ with pairwise distinct
coordinates agrees with the number $\chi_r^{\bfs{a}}$. In the sequel
we follow a different approach, considering the incidence variety
consisting of the set of triples
$(\bfs{b},b_0,\alpha_1\klk\alpha_r)$ with
$\chi_r:=\{\alpha_1\klk\alpha_r\}$ and
$(f_{\boldsymbol{b}}+b_0)|_{\chi_r}\equiv 0$.

Fix $r$ with $d-s+1\leq r\leq d$. Let
$T,T_1,\ldots,T_r,B_{d-s-1},\ldots,B_1,B_0$ be new indeterminates
over $\cfq$, let $\boldsymbol{T}:=(T_1,\ldots,T_r)$,
$\boldsymbol{B}:=(B_{d-s-1},\ldots,B_{1})$ and
$\bfs{B}_0:=(\bfs{B},B_0)$, and let $F\in\fq[\boldsymbol{B}_0,T]$ be
the polynomial defined in the following way:
\begin{equation}\label{eq: definition of gB}
F:=T^d+ \displaystyle\sum_{i=d-s}^{d-1}a_i T^i +
\displaystyle\sum_{i=1}^{d-s-1}B_i T^i +B_0.\
\end{equation}
Finally, we consider the affine quasi-$\fq$-variety
$\Gamma_{r}\subset \A^{d-s+r}$ defined as follows:
$$
\Gamma_{r}:=\{(\boldsymbol{b}_{0}, \boldsymbol{\alpha})
\in\A^{d-s}\times\A^r:F({\boldsymbol{b}_{0},\alpha_j})=0\ (1\leq
j\leq r),\ \alpha_i\neq\alpha_j\ (1\le i<j\le r)\}.
$$
Our next result relates the number $|\Gamma_{r}(\fq)|$ of
$q$--rational points of $\Gamma_r$ with $\chi_r^{\bfs{a}}$.
\begin{lemma}\label{lemma: relacion entre gamma y chi}
Let $r$ be an integer with  $d-s+1\leq
r\leq d$. Then the following identity holds:
$$\frac{|\Gamma_{r}(\fq)|}{r!}=\chi_r ^{\boldsymbol{a}}.$$
\end{lemma}
\begin{proof}
Let $(\boldsymbol{b}_{0}, \boldsymbol{\alpha})$ be an arbitrary
point of $\Gamma_{r}(\fq)$ and let
$\sigma:\{1,\dots,r\}\to\{1,\dots,r\}$ be an arbitrary permutation.
Let $\sigma(\boldsymbol{\alpha})$ be the image of
$\boldsymbol{\alpha}$ by the linear mapping induced by $\sigma$.
Then it is easy to see that $\big(\boldsymbol{b}_{0},
\sigma(\boldsymbol{\alpha})\big)$ is also a point of
$\Gamma_{r}(\fq)$. Furthermore,
$\sigma(\boldsymbol{\alpha})=\boldsymbol{\alpha}$ if and only if
$\sigma$ is the identity permutation. This shows that
$\mathbb{S}_r$, the symmetric group of $r$ elements, acts over the
set $\Gamma_{r}(\fq)$ and each orbit under this action has $r!$
elements.

The orbit of an arbitrary point $(\boldsymbol{b}_{0},
\boldsymbol{\alpha})\in \Gamma_{r}(\fq)$ uniquely determines a
polynomial $F({\boldsymbol{b}_0, T)}\in\fq[T]$ and a set
$\chi_r:=\{\alpha_1,\dots,\alpha_r\}\subset\fq$ with $|\chi_r|=r$
and $F({\boldsymbol{b}_0,T})|_{\chi_r}\equiv 0$. On the other hand,
each subset $\chi_r:=\{\alpha_1,\dots,\alpha_r\}$ as in the
statement of Theorem \ref{th: combinatorial reduction mean}
determines a unique $\bfs{b}_0\in\fq^{d-s}$ such that the polynomial
$F(\bfs{b}_0,T)$ vanishes on $\chi_r$, and thus a unique orbit as
above. This implies that the number of orbits of $\Gamma_{r}(\fq)$
is equal to $\chi_r ^{\boldsymbol{a}}$ and finishes the proof of the
lemma.
\end{proof}

In order to estimate the quantity $|\Gamma_{r}(\fq)|$ we shall
consider the Zariski closure $\mathrm{cl}(\Gamma_r)$ of
$\Gamma_r\subset \A^{d-s+r}$. In order to determine equations
defining $\mathrm{cl}(\Gamma_r)$, we shall use the following
notation. Let $T,X_1\klk X_{l+1}$ be indeterminates over $\cfq$ and
let $f\in\cfq[T]$ be a polynomial of degree at most $l$. For
notational convenience, we define the 0th divided difference
$\Delta^0f\in\cfq[X_1]$ of $f$ as $\Delta^0f:=f(X_1)$. Further, for
$1\le i\le l$ we define the $i$th divided difference
$\Delta^if\in\cfq[X_1\klk X_{i+1}]$ of $f$ as
$$\Delta^if(X_1,\ldots,X_{i+1})=\dfrac{\Delta^{i-1}f(X_1,\ldots,X_i)-
\Delta^{i-1}f(X_1,\ldots,X_{i-1},X_{i+1})}{X_i-X_{i+1}}.$$

With these notations, we define the following affine $\fq$--variety
$\Gamma^*_{r}\subset \A^{d-s+r}$:
$$
\Gamma_{r}^*:=\{(\boldsymbol{b}_{0},\boldsymbol{\alpha})
\in\A^{d-s}\times\A^r:\Delta^{i-1}F(\boldsymbol{b}_0,
\alpha_1\klk\alpha_i)=0\ (1\leq i\leq r) \},$$
where $\Delta^{i-1}F(\boldsymbol{b}_{0},T_1,\ldots,T_i)$  denotes
the $(i-1)$--divided difference of $F(\boldsymbol{b}_{0},T)$ $\in
\cfq[T]$. The relation between the varieties $\Gamma_{r}$ and
$\Gamma_{r}^*$ is expressed in the following result.
\begin{lemma}\label{lemma: relacion gamma_r y gamma_r estrella}
With notations and assumptions as above, we have the following
identity:
\begin{equation}\label{eq: relacion gamma_r y gamma_r estrella}
\Gamma_{r}=\Gamma_{r}^*\cap\{(\bfs{b}_0,\bfs{\alpha}):\alpha_i\neq\alpha_j\
(1\le i<j\le r)\}.
\end{equation}
\end{lemma}
\begin{proof}
Let $(\boldsymbol{b}_{0}, \boldsymbol{\alpha})$ be a point  of
$\Gamma_r$. By the definition of the divided differences of
$F(\boldsymbol{b}_0,T)$ we easily conclude that
$(\boldsymbol{b}_{0}, \boldsymbol{\alpha})\in \Gamma_r^*$. On the
other hand, let $(\boldsymbol{b}_{0}, \boldsymbol{\alpha})$ be a
point belonging to the set of the right--hand side of (\ref{eq:
relacion gamma_r y gamma_r estrella}). We claim that
$F(\boldsymbol{b}_0,\alpha_k)=0$ for $1\le k\le r$. We observe that
$F(\boldsymbol{b}_0,\alpha_1)=
\Delta^0F(\boldsymbol{b}_{0},\alpha_1)=0$. Arguing inductively,
suppose that we have $F(\boldsymbol{b}_0,\alpha_1)=\cdots
=F(\boldsymbol{b}_0,\alpha_{i-1})=0$. By definition we conclude that
the quantity $\Delta^{i-1}F(\boldsymbol{b}_{0},
\alpha_1\cdots\alpha_i)$ can be expressed as a linear combination
with nonzero coefficients of the differences
$F(\boldsymbol{b}_0,\alpha_{j+1})-F(\boldsymbol{b}_0,\alpha_j)$ with
$1\le j\le i-1$. Therefore, combining the inductive hypothesis with
the fact that $\Delta^{i-1}F(\boldsymbol{b}_{0},
\alpha_1\klk\alpha_i)=0$, we easily conclude
$F(\boldsymbol{b}_0,\alpha_i)=0$, finishing thus the proof of the
claim.
\end{proof}
%
%
\section{Geometry of the variety $\Gamma_r^*$}
\label{sec: geometry of Gamma_r estrella}
From now on we assume that the characteristic $p$ of $\fq$ is
strictly greater than 2. This section is devoted to establish a
number of facts concerning the geometry of the affine $\fq$--variety
$\Gamma_r^*$. We first show that the defining polynomials of
$\Gamma_r^*$ form a regular sequence, which in particular allows us
to determine the dimension of $\Gamma_r^*$. Then we analyze the
singular locus $\Gamma_r^*$, showing that it has codimension at
least 2 in $\Gamma_r^*$. Finally, we show a number of results
concerning the projective closure $\mathrm{pcl}(\Gamma_r^*)$ of
$\Gamma_r^*$ and the set of points of $\mathrm{pcl}(\Gamma_r^*)$ at
infinity. The final outcome is that both $\mathrm{pcl}(\Gamma_r^*)$
and the set of points of $\mathrm{pcl}(\Gamma_r^*)$ at infinity are
normal complete intersections, which will allow us to obtain a
suitable estimate on the number of $q$--rational points of
$\Gamma_r^*$.
\begin{lemma}\label{lemma: Gamma r is set-theoret complete intersection}
$\Gamma_{r}^{*}$ is a (set-theoretic) complete intersection of
dimension $d-s$.
\end{lemma}
\begin{proof}
Consider the graded lexicographic order of
$\cfq[\boldsymbol{B}_{0},\boldsymbol{T}]$ with $T_r>\cdots>T_1
>B_{d-s-1} >\cdots>B_{0}$. It is easy to see that for each $i$ the
polynomial $\Delta^{i-1}F(\boldsymbol{B}_0,T_1,\ldots,T_i)$ has
degree $d-i+1$ in the variables $\boldsymbol{T}$ and the monomial
$T_i^{d-i+1}$ arises in the dense representation of such a
polynomial with nonzero coefficient. We deduce that the leading term
of $\Delta^{i-1}F(\boldsymbol{B}_{0},T_1,\ldots,T_i)$ is
$T_i^{d-i+1}$ for $1\leq i\leq r$ in the monomial order defined
above. Hence the leading terms of
$\Delta^{i-1}F(\boldsymbol{B}_{0},T_1,\ldots,T_i)$ $(1\leq i \leq
r)$ are relatively prime and thus they form a Gr\"obner basis of the
ideal $\mathcal{J}$ that they generate (see, e.g., \cite[\S 2.9,
Proposition 4]{CoLiOS92}), the initial ideal of $\mathcal{J}$ being
generated by $\{T_i^{d-i+1}:1\leq i\leq r\}$. Furthermore, since
$\{T_i^{d-i+1}:1\leq i \leq r\}$ form a regular sequence of
$\cfq[\boldsymbol{B}_0,\boldsymbol{T}]$, from, e.g.,
\cite[Proposition 15.15]{Eisenbud95}, we conclude that
$\{\Delta^{i-1}F(\boldsymbol{B}_0,T_1,\ldots,T_i):1\leq i \leq r\}$
also form a regular sequence. This finishes the proof of the lemma.
\end{proof}
%
%
\subsection{The dimension of the singular locus of $\Gamma_{r }^*$ and consequences}
As asserted above, we shall study the dimension of the singular
locus of $\Gamma_r^{*}$. Our aim is to show that such a singular
locus has codimension at least $2$ in $\Gamma_r^{*}$.

We start with following simple criteria of nonsingularity.
\begin{lemma}\label{lemma: jacobian_F full rank implies nonsingular}
Let $J_F\in\fq[\boldsymbol{B}_{0},\bfs{T}]^{r\times(d-s+r)}$ be the
Jacobian matrix of the polynomials $F(\boldsymbol{B}_0,T_i)$ $(1
\leq i \leq r)$ with respect to $\boldsymbol{B}_{0},\boldsymbol{T}$
and let $(\boldsymbol{b}_{0},\boldsymbol{\alpha})$ be an arbitrary
point of $\Gamma_r^*$. If $J_F$ evaluated at
$(\boldsymbol{b}_{0},\boldsymbol{\alpha})$ has full rank, then
$(\boldsymbol{b}_{0},\boldsymbol{\alpha})$ is a nonsingular point of
$\Gamma_r^*$.
\end{lemma}
\begin{proof}
Considering the Newton form of the polynomial interpolating
$F(\boldsymbol{b}_0,T)$ at $\alpha_1\klk\alpha_r$ we easily deduce
that $F(\boldsymbol{b}_{0},\alpha_i)=0$ for $1 \leq i \leq r$. This
shows that $F(\boldsymbol{B}_{0},T_i)$ vanishes on $\Gamma_r^*$ for
$1\le i\le r$. As a consequence, any element of the tangent space
$\mathcal{T}_{(\boldsymbol{b}_{0},\boldsymbol{\alpha})}\Gamma_r^*$
of $\Gamma_r^*$ at $(\boldsymbol{b}_{0},\boldsymbol{\alpha})$
belongs to the kernel of the Jacobian matrix
$J_F(\boldsymbol{b}_{0},\boldsymbol{\alpha})$.

By hypothesis, the $\big(r\times (d-s+r)\big)$--matrix
$J_F(\boldsymbol{b}_{0},\boldsymbol{\alpha})$ has full rank $r$, and
thus, its kernel has dimension $d-s$. We conclude that the tangent
space
$\mathcal{T}_{(\boldsymbol{b}_{0},\boldsymbol{\alpha})}\Gamma_r^*$
has dimension at most $d-s$. Since $\Gamma_r^*$ is equidimensional
of dimension $d-s$, it follows that
$(\boldsymbol{b}_{0},\boldsymbol{\alpha})$ is a nonsingular point of
$\Gamma_r^*$.
\end{proof}

Let $(\boldsymbol{b}_{0},\boldsymbol{\alpha})$ be an arbitrary point
of $\Gamma_r^*$ with
$\boldsymbol{\alpha}:=(\alpha_1,\ldots,\alpha_r)$, and let
$f_{\boldsymbol{b}_0}:=F(\boldsymbol{b}_{0},T)$. Then the Jacobian
matrix $J_F$ evaluated at $(\boldsymbol{b}_{0},\boldsymbol{\alpha})$
has the following form:
$$J_F(\boldsymbol{b}_{0},\boldsymbol{\alpha}):=\left(\begin{array}{cccccccc}
  \alpha_1^{d-s-1}  & \ldots  & \alpha_1  & 1  & f_{\boldsymbol{b}_0}'(\alpha_1)& 0 & \cdots & 0\\
  \vdots            && \vdots    &\vdots  &  0   & \ddots             & \ddots& \vdots\\
  \vdots            && \vdots    &\vdots  &  \vdots   &  \ddots       & \ddots    &0  \\
   \alpha_r^{d-s-1} & \ldots  & \alpha_r  & 1 &0 & \cdots& 0 & f_{\boldsymbol{b}_0}'(\alpha_r)
\end{array}\right).
$$
We observe that, if all the roots in $\cfq$ of
$f_{\boldsymbol{b}_0}$ are simple, then
$J_F(\boldsymbol{b}_{0},\boldsymbol{\alpha})$ has full rank and
$(\boldsymbol{b}_{0},\boldsymbol{\alpha})$ is a regular point of
$\Gamma_r^*$. Therefore, in order to prove that the singular locus
of $\Gamma_r^*$ is a subvariety of codimension at least $2$, it
suffices to consider the set of points
$(\boldsymbol{b}_{0},\boldsymbol{\alpha}) \in\Gamma_r^*$ for which
at least one coordinate of $\bfs{\alpha}$ is a multiple root of
$f_{\boldsymbol{b}_0}$. In particular, $f_{\boldsymbol{b}_0}$ must
have multiple roots. We start considering the ``extreme'' case where
$f'_{\boldsymbol{b}_0}$ is the zero polynomial.
\begin{lemma}\label{lemma: f'=0}
If $d-s\geq 3$, then the set $\mathcal{W}_1$ of  points
$(\boldsymbol{b}_{0},\boldsymbol{\alpha})\in\Gamma_{r}^*$ with
$f'_{\boldsymbol{b}_0}=0$ is contained in a subvariety of
codimension 2 of $\Gamma_r^*$.
\end{lemma}
\begin{proof}
Consider the morphism of $\fq$-varieties defined as follows:
\begin{equation}\label{eq: morfismo finito - mean}
\begin{array}{rccl}
\Psi_r:& {\Gamma}_{r}^{*}& \longrightarrow& \A^{d-s}\\
   &(\boldsymbol{b}_{0},\boldsymbol{\alpha})& \mapsto&\boldsymbol{b}_{0}.
\end{array}
\end{equation}
We claim that $\Psi_r$ is a finite morphism. In order to prove this
claim, it is enough to show that the coordinate function $t_j$ of
$\cfq[\Gamma_r^*]$ defined by $T_j$ satisfies a monic equation with
coefficients in $\cfq[\bfs{B}_0]$ for $1\le j\le r$. Since the
polynomial $F(\bfs{B}_0,T_j)$ vanishes on $\Gamma_r^*$ and is a
monic element of $\cfq[\bfs{B}_0][T_j]$, it provides the monic
equation annihilating $t_j$ that we are looking for.

Since $d-s\geq 3$, then $d-s-1\ge 2$ and the condition
$f'_{\boldsymbol{b}_0}=f_{\bfs{a}}'+\sum_{j=1}^{d-s-1}jb_jT^{j-1}=0$
implies $b_1=b_2=0$. It follows that the set of points
$(\boldsymbol{b}_{0},\boldsymbol{\alpha})\in\Gamma_{r}^*$ with
$f'_{\boldsymbol{b}_0}=0$ is a subset of
$\Psi_r^{-1}(\mathcal{Z}_{1,2})$, where
$\mathcal{Z}_{1,2}\subset\A^{d-s}$ is the variety of dimension
$d-s-2$ defined by the equations $B_1=B_2=0$.  Taking into account
that $\Psi_r$ is a finite morphism we deduce that
$\Psi_r^{-1}(\mathcal{Z}_{1,2})$ has dimension $d-s-2$.
\end{proof}

In what follows we shall assume that $f'_{\boldsymbol{b}_0}$ is
nonzero and $f_{\boldsymbol{b}_0}$ has multiple roots. We analyze
the case where exactly one of the coordinates of $\bfs{\alpha}$ is a
multiple root of $f_{\boldsymbol{b}_0}$.
\begin{lemma}\label{lemma: only one multiple root - mean}
Suppose that $f'_{\boldsymbol{b}_0}\not=0$ and there exists a unique
coordinate $\alpha_i$ of $\bfs{\alpha}$ which is a multiple root of
$f_{\boldsymbol{b}_0}$. Then $(\boldsymbol{b}_0,\bfs{\alpha})$ is a
regular point of $\Gamma_r^{*}$.
\end{lemma}
\begin{proof}
Assume without loss of generality that $\alpha_1$ is the only
multiple root of $f_{\boldsymbol{b}_0}$ among the coordinates of
$\bfs{\alpha}$. According to Lemma \ref{lemma: jacobian_F full rank
implies nonsingular}, it suffices to show that the Jacobian matrix
$J_F(\boldsymbol{b}_0,\bfs{\alpha})$ has full rank. For this
purpose, we observe that the $(r\times r)$--submatrix of
$J_F(\boldsymbol{b}_0,\bfs{\alpha})$ consisting of the $(d-s)$th
column and the last $r-1$ columns of
$J_F(\boldsymbol{b}_0,\bfs{\alpha})$, namely
$$\hat{J}_F(\boldsymbol{b}_0,\bfs{\alpha}):=\left(\begin{array}{cccccccc}
 1                & 0& 0                      & \cdots & 0\\
 1                & f_{\boldsymbol{b}_0}'(\alpha_2)& 0                      & \cdots & 0\\
\vdots  &  0   & \ddots       & \ddots   &\vdots \\
\vdots  &  \vdots   &  \ddots      &  \ddots  &0 \\
 1 &0 & \cdots& 0& f_{\boldsymbol{b}_0}'(\alpha_r)
\end{array}\right),
$$
%
is nonsingular. Indeed, by hypothesis $\alpha_i$ is a simple root of
$f'_{\boldsymbol{b}_0}$, which implies
$f'_{\boldsymbol{b}_0}(\alpha_i)\neq 0$ for $i\ge 2$. We conclude
that $J_F(\boldsymbol{b}_0,\bfs{\alpha})$ has full rank.
\end{proof}

The next case to be discussed is the one when two distinct multiple
roots of $f_{\boldsymbol{b}_0}$ occur among the coordinates of
$\bfs{\alpha}$.
\begin{lemma}\label{lemma: two distinct multiple roots - mean}
Let $\mathcal{W}_2$ denote the set of points
$(\boldsymbol{b}_0,\bfs{\alpha})\in \Gamma_{r}^{*}$ for which there
exist $1\le i<j\le r$ such that $\alpha_i\not=\alpha_j$ and
$\alpha_i,\alpha_j$ are multiple roots of $f_{\boldsymbol{b}_0}$.
Then $\mathcal{W}_2$ is contained in a subvariety of codimension 2
of $\Gamma_{r}^{*}$.
\end{lemma}
\begin{proof}
Let $(\boldsymbol{b}_0,\bfs{\alpha})$ be an arbitrary point of
$\mathcal{W}_2$. We may assume without loss of generality that
$f'_{\boldsymbol{b}_0}\not=0$ holds. Since $f_{\boldsymbol{b}_0}$
has at least two distinct multiple roots, the greatest common
divisor of $f_{\boldsymbol{b}_0}$ and $f'_{\boldsymbol{b}_0}$ has
degree at least $2$. Hence we have:
$$\mathrm{Res}(f_{\boldsymbol{b}_0},f'_{\boldsymbol{b}_0})=
\mathrm{Subres}(f_{\boldsymbol{b}_0},f'_{\boldsymbol{b}_0})=0,$$
where $\mathrm{Res}(f_{\boldsymbol{b}_0},f'_{\boldsymbol{b}_0})$ and
$\mathrm{Subres}(f_{\boldsymbol{b}_0},f'_{\boldsymbol{b}_0})$ denote
the resultant and the first--order subresultant of
$f_{\boldsymbol{b}_0}$ and $f'_{\boldsymbol{b}_0}$ respectively.
Furthermore, since $f_{\boldsymbol{b}_0}$ has degree $d$, by basic
properties of resultants and subresultants it follows that
\begin{eqnarray*}
\mathrm{Res}(f_{\boldsymbol{b}_0},
f'_{\boldsymbol{b}_0})=\mathrm{Res}(F(\boldsymbol{B}_{0},T),
\Delta^1F(\boldsymbol{B}_0,T,T),T)|_{\boldsymbol{B}_0=\boldsymbol{b}_0}\quad\ \\
\mathrm{Subres}(f_{\boldsymbol{b}_0},
f'_{\boldsymbol{b}_0})=\mathrm{Subres}(F(\boldsymbol{B}_{0},T),
\Delta^1F(\boldsymbol{B}_{0},T,T),T))|_{\boldsymbol{B}_0=\boldsymbol{b}_0},
\end{eqnarray*}
where $\mathrm{Res}(F(\boldsymbol{B}_{0},T),
\Delta^1F(\boldsymbol{B}_{0},T,T),T)$ and
$\mathrm{Subres}(F(\boldsymbol{B}_{0},T),
\Delta^1F(\boldsymbol{B}_{0},T,T),T)$ are the resultant and the
first--order subresultant of $F(\boldsymbol{B}_{0},T)$ and
$\Delta^1F(\boldsymbol{B}_{0},T,T)$ with respect to $T$. As a
consequence,
$\mathcal{W}_2\cap\Gamma_r^*\subset\Psi_r^{-1}(\mathcal{Z}_2)$,
where $\Psi_r$ is the morphism of (\ref{eq: morfismo finito - mean})
and $\mathcal{Z}_2$ is the subvariety of $\A^{d-s}$ defined by the
equations
\begin{equation}\label{eq: res y subres} \mathrm{Res}(F(\boldsymbol{B}_{0},T),
\Delta^1F(\boldsymbol{B}_{0},T,T),T)=
\mathrm{Subres}(F(\boldsymbol{B}_{0},T)
,\Delta^1F(\boldsymbol{B}_{0},T,T),T)=0.
\end{equation}

We first observe that
$\mathcal{R}:=\mathrm{Res}(F(\boldsymbol{B}_{0},T),
\Delta^1F(\boldsymbol{B}_{0},T,T),T)$ is a nonzero polynomial
because $F(\boldsymbol{B}_0,T)$ is a separable element of
$\fq[\boldsymbol{B}_0][T]$. We claim that the first--order
subresultant
$\mathcal{S}_1:=\mathrm{Subres}(F(\boldsymbol{B}_{0},T),
\Delta^1F(\boldsymbol{B}_{0},T,T),T))$ is a nonzero polynomial.
Indeed, if $p$ does not divide $d(d-1)$, then the nonzero term
$d(d-1)^{d-2}B_1^{d-2}$ occurs in the dense representation of
$\mathcal{S}_1$. On the other hand, if $p$ divides $d(d-1)$, since
$p>2$, the nonzero term $2\,(-1)^d(d-2)^{d-2}B_2^{d-1}$ arises in
the dense representation of $\mathcal{S}_1$.

We claim that the polynomials arising in (\ref{eq: res y subres})
form a regular sequence in $\fq[\boldsymbol{B}_0]$. Indeed, since
$p>2$, we have that $\mathcal{R}$ is an irreducible element of
$\cfq[\boldsymbol{B}_0]$ (see Theorem \ref{th: irred discrim}
below). If $\mathcal{S}_1$ is a zero divisor in the quotient ring
$\cfq[\boldsymbol{B}_0]/(\mathcal{R})$, then $\mathcal{S}_1$ must be
a multiple of $\mathcal{R}$ in $\cfq[\boldsymbol{B}_0]$ which is
impossible because $\max\{\deg_{B_1}\mathcal{R},
\deg_{B_2}\mathcal{R}\}=d$, while $\max\{\deg_{B_1}\mathcal{S}_1,
\deg_{B_2}\mathcal{S}_1\}\le d-1$. It follows that $\dim
\mathcal{Z}_2=d-s-2$, and hence $\dim
\Psi_r^{-1}(\mathcal{Z}_2)=d-s-2$. Therefore, $\mathcal{W}_2$ is
contained in a subvariety of $\Gamma_r^*$ of codimension 2.
\end{proof}

It remains to consider the case where only one multiple root of
$f_{\boldsymbol{b}_{0}}$ occurs among the coordinates of
$\bfs{\alpha}$, but there are at least two distinct coordinates of
$\bfs{\alpha}$ taking such a value. Then we have either that all the
remaining coordinates of $\bfs{\alpha}$ are simple roots of
$f_{\boldsymbol{b}_{0}}$, or there exists at least a third
coordinate whose value is the same multiple root. Our next result
deals with the first of these two cases.
\begin{lemma}\label{lemma: repeated multiple root I - mean} Let
$(\boldsymbol{b}_{0}, \boldsymbol{\alpha})\in \Gamma_{r}^*$ be a
point satisfying the following conditions:
\begin{itemize}
\item there exist $1\le i<j\le r$ such that $\alpha_i=\alpha_j$ and
$\alpha_i$ is a multiple root of $f_{\boldsymbol{b}_{0}}$;
\item for any $k\notin\{i,j\}$, $\alpha_k$ is a simple root of
$f_{\boldsymbol{b}_{0}}$.
\end{itemize}
Then $(\boldsymbol{b}_{0}, \boldsymbol{\alpha})$ is regular point of
$\Gamma_{r}^*$.
\end{lemma}
\begin{proof} The argument is similar to that of the proof of Lemma
\ref{lemma: jacobian_F full rank implies nonsingular}. Assume
without loss of generality that $i=1$ and $j=2$. We observe that the
polynomials $\Delta^1F(\boldsymbol{B}_0,T_1,T_2)$ and
$F(\boldsymbol{B}_0,T_i)$ $(2\le i\le r)$ vanish on $\Gamma_r^*$.
Therefore, the tangent space
$\mathcal{T}_{(\boldsymbol{b}_{0},\boldsymbol{\alpha})}\Gamma_r^*$
of $\Gamma_r^*$ at $(\boldsymbol{b}_{0},\boldsymbol{\alpha})$ is
included in the kernel of the Jacobian matrix
$J_{\Delta,F}(\boldsymbol{b}_{0},\boldsymbol{\alpha})$ of $\Delta^1
F(\boldsymbol{B}_0,T_1,T_2)$ and $F(\boldsymbol{B}_0,T_i)$ $(2\le
i\le r)$ with respect to $\boldsymbol{B}_0,\bfs{T}$. If
$J_{\Delta,F}(\boldsymbol{b}_{0},\boldsymbol{\alpha})$ has full rank
$r$, then its kernel has dimension $d-s$. This implies that $\dim
\mathcal{T}_{(\boldsymbol{b}_{0}, \boldsymbol{\alpha})}\Gamma_r^*
\le d-s$, which proves that $(\boldsymbol{b}_{0},
\boldsymbol{\alpha})$ is regular point of $\Gamma_r^*$.

It is easy to see that $\frac{\partial\, \Delta^1F}{\partial
B_0}(\boldsymbol{b}_{0},\alpha_1,\alpha_1)=0$ and $\frac{\partial\,
\Delta^1F}{\partial
B_i}(\boldsymbol{b}_{0},\alpha_1,\alpha_1)=i\alpha_1^{i-1}$ for
$i\ge 1$. Therefore, we have
$$J_{\Delta,F}(\boldsymbol{b}_{0},\boldsymbol{\alpha}):=\left(\begin{array}{ccccccccc}
(d-s-1)\alpha_1^{d-s-2} & \ldots  & 1  & 0  & * &* & 0 & \cdots & 0\\
\alpha_2^{d-s-1}  & \ldots  & \alpha_2  & 1  & 0 & 0 & 0 & \cdots & 0\\
\alpha_3^{d-s-1}  & \ldots  & \alpha_3  & 1  & 0 & 0   & \gamma_3 & \ddots& \vdots\\
  \vdots            && \vdots    &\vdots  &  \vdots   &\vdots&  \ddots       & \ddots    &0  \\
   \alpha_r^{d-s-1} & \ldots  & \alpha_r  & 1 &0 & 0& \cdots & 0 & \gamma_r
\end{array}\right),
$$
where $\gamma_i:=f_{\boldsymbol{b}_0}'(\alpha_i)$ for $i\ge 3$.
Since $\alpha_i$ is a simple root of $f_{\boldsymbol{b}_0}$ for
$i\ge 3$, it follows that $\gamma_i\not=0$, which implies that
$J_{\Delta,F}(\boldsymbol{b}_{0},\boldsymbol{\alpha})$ has rank $r$.
This finishes the proof of the lemma.
\end{proof}

Finally, we consider the set of points $(\boldsymbol{b}_{0},
\boldsymbol{\alpha})\in\Gamma_r^*$ where the value of at least three
distinct coordinates of $\bfs{\alpha}$ is the same multiple root of
$f_{\boldsymbol{b}_{0}}$.
\begin{lemma}\label{lemma: repeated multiple root II - mean}
Let $\mathcal{W}_3\subset\Gamma_{r}^*$ be the set of points
$(\boldsymbol{b}_{0}, \boldsymbol{\alpha})$ for which there exist
$1\le i<j<k\le r$ such that $\alpha_i=\alpha_j=\alpha_k$ and
$\alpha_i$ is a multiple root of $f_{\boldsymbol{b}_{0}}$. If either
$d-s \geq 4$ and $p>3$, or $d-s\ge 6$ and $p=3$, then
$\mathcal{W}_3$ is contained in a codimension--2 subvariety of
$\Gamma_r^*$.
\end{lemma}
\begin{proof}
Let $(\boldsymbol{b}_{0}, \boldsymbol{\alpha})$ be an arbitrary
point of $\mathcal{W}_3$. Without loss of generality we may assume
that $\alpha_1=\alpha_2=\alpha_3$ is the multiple root of
$f_{\boldsymbol{b}_{0}}$. Taking into account that
$(\boldsymbol{b}_{0}, \boldsymbol{\alpha})$ satisfies the equations
$$
F(\boldsymbol{B}_{0},T_1)=\Delta F(\boldsymbol{B}_{0},T_1,T_2)=
\Delta^{2}F(\boldsymbol{B}_{0},T_1,T_2,T_3)=0,
$$
we conclude that $\alpha_1$ is a common root of the polynomials
$f_{\boldsymbol{b}_0}$, $\Delta F(\boldsymbol{b}_{0},T,T)$ and
$\Delta^2F(\boldsymbol{b}_{0},T,T,T)$.

Under the hypotheses on $d$, $s$ and $p$ of the statement of the
lemma, it is easy to see that there exists $j$ with $2<j\le d-s-1$
such that $j(j-1)\not\equiv 0\mod p$ holds. Therefore, the condition
$2\Delta^2F(\boldsymbol{b}_{0},T,T,T)=
f_{\bfs{a}}''+\sum_{j=2}^{d-s-1}j(j-1)b_jT^{j-2}=0$ implies
$b_2=b_j=0$. Then the set $\mathcal{W}_3'$ of points
$(\boldsymbol{b}_{0},\boldsymbol{\alpha})\in\Gamma_r^*$ such that
$\Delta^2F(\boldsymbol{b}_{0},T,T,T)=0$ holds is contained in
$\Psi_r^{-1}(\mathcal{Z}_{2,j})$, where
$\mathcal{Z}_{2,j}\subset\A^{d-s}$ is the variety of dimension
$d-s-2$ defined by the equations $B_2=B_j=0$. Since $\Psi_r$ is a
finite morphism we deduce that $\Psi_r^{-1}(\mathcal{Z}_{2,j})$ has
dimension $d-s-2$. Therefore, we may assume without loss of
generality that $\Delta^2F(\boldsymbol{b}_{0},T,T,T)$ is a nonzero
polynomial.

Suppose that $p$ does not divide $d$. Then $f_{\boldsymbol{b}_0}$
and $f'_{\boldsymbol{b}_0}$ are nonzero polynomials of degree $d$
and $d-1$ respectively. Hence, by elementary properties of
resultants we deduce that
\begin{eqnarray*}
\mathrm{Res}(f_{\boldsymbol{b}_0},f'_{\boldsymbol{b}_0})=
\mathrm{Res}\big(F(\boldsymbol{B}_{0},T),\Delta^1
F(\boldsymbol{B}_{0},T,T),T\big)\big|_{\boldsymbol{B}_0=
\boldsymbol{b}_0},\qquad\quad\ \\
\mathrm{Res}(f'_{\boldsymbol{b}_0},\Delta^{2}f_{\boldsymbol{b}_0})=
\mathrm{Res}\big(\Delta^1
F(\boldsymbol{B}_{0},T,T),\Delta^{2}F(\boldsymbol{B}_{0},T,T,T),
T\big)\big|_{\boldsymbol{B}_0=\boldsymbol{b}_0}.
\end{eqnarray*}
We conclude that $(\mathcal{W}_3\setminus
\mathcal{W}_3')\cap\Gamma_r^*\subset \Psi_r^{-1}(\mathcal{Z}_3)$,
where $\Psi_r$ is the morphism of (\ref{eq: morfismo finito - mean})
and $\mathcal{Z}_3$ is the subvariety of $\A^{d-s}$ defined by the
equations:
\begin{align*}
\mathrm{Res}\big(F(\boldsymbol{B}_{0},T),
\Delta^1F(\boldsymbol{B}_{0},T,T),T\big)=0,\\
\mathrm{Res}\big(\Delta^1F(\boldsymbol{B}_{0},T,T),
\Delta^2F(\boldsymbol{B}_{0},T,T,T),T\big)=0.
\end{align*}
According to Theorem \ref{th: irred discrim} below,
$\mathrm{Res}\big(F(\boldsymbol{B}_{0},T),\Delta^1
F(\boldsymbol{B}_0,T,T),T\big)$ is an irreducible element of
$\fq[\boldsymbol{B}_0]$ having degree $d-1$ in $B_0$. On the other
hand, the nonzero polynomial $\mathrm{Res}\big(\Delta^1
F(\boldsymbol{B}_{0},T,T),
\Delta^2F(\boldsymbol{B}_{0},T,T,T),T\big)$ has degree 0 in $B_0$.
As a consequence, both polynomials form a regular sequence in
$\cfq[\boldsymbol{B}_{0}]$, which shows that $\mathcal{Z}_3$ has
codimension 2 in $\A^{d-s}$. This proves that
$\Psi_r^{-1}(\mathcal{Z}_3)$ is a codimension 2 subvariety of
$\Gamma_r^*$.

Now suppose that $p$ divides $d$. If there exists $l$ such that
$la_l\not\equiv 0\mod p$, then $f'_{\boldsymbol{b}_0}$ and $\Delta^1
F({\boldsymbol{B}_{0},T,T})$ are of the same degree in the variable
$T$ and the argument above follows {\em mutatis mutandis}.

On the other hand, if $la_l\equiv 0\mod p$ for $d-s\leq l \leq d-1$,
we have two possibilities, according to whether or not
$(d-s-1)b_{d-s-1}=0$. If $(d-s-1)b_{d-s-1}\neq 0$, then
$f'_{\boldsymbol{b}_0}$ and $\Delta^1 F({\boldsymbol{B}_{0},T,T})$
are of the same degree in the variable $T$ and the previous argument
follows. If $b_{d-s-1}=0$, then we have that $(\boldsymbol{b}_{0},
\boldsymbol{\alpha})\in \Psi_r^{-1}(\mathcal{Z}_4)$, where
$\mathcal{Z}_4$ is the subvariety of $\A^{d-s}$ defined by
$$\mathrm{Res}\big(F({\boldsymbol{B}_{0},T}),
\Delta^1 F({\boldsymbol{B}_{0},T,T}),T\big)=B_{d-s-1}=0.$$
It is easy to see that the polynomials defining these equations form
a regular sequence of $\cfq[\boldsymbol{B}_{0}]$, which shows that
$\Psi_r^{-1}(\mathcal{Z}_4)$ is a subvariety of codimension 2 of
$\Gamma_{r}^*$. Finally, if $p$ divides $d-s-1$, then $p$ does not
divide $d-s-2$, and we can repeat previous arguments considering the
cases $b_{d-s-2}=0$ and  $b_{d-s-2}\neq 0$. This finishes the proof
of the lemma.
\end{proof}
Now we are in position of proving the main result of this section.
According to Lemmas \ref{lemma: f'=0}, \ref{lemma: only one multiple
root - mean}, \ref{lemma: two distinct multiple roots - mean},
\ref{lemma: repeated multiple root I - mean} and \ref{lemma:
repeated multiple root II - mean}, the set of singular points of
$\Gamma_r^*$ is contained in the set $\mathcal{W}_1\cup
\mathcal{W}_2\cup\mathcal{W}_3$, where $\mathcal{W}_1$,
$\mathcal{W}_2$ and $\mathcal{W}_3$ are defined in the statement of
Lemmas \ref{lemma: f'=0}, \ref{lemma: two distinct multiple roots -
mean} and \ref{lemma: repeated multiple root II - mean}. Since each
set $\mathcal{W}_i$ is contained in codimension--2 subvariety of
$\Gamma_r^*$, we obtain the following result.
\begin{theorem}\label{th: codimension singular locus Gamma r}
If either $d-s \geq 4$ and $p>3$, or $d-s\ge 6$ and $p=3$, then the
singular locus of $\Gamma_r^*$ has codimension at least $2$ in
$\Gamma_r^*$.
\end{theorem}

We finish the section by discussing a few consequences of the
analysis underlying the proof of Theorem \ref{th: codimension
singular locus Gamma r}.
\begin{corollary}\label{coro: J is radical}
Let assumptions be as in Theorem \ref{th: codimension singular locus
Gamma r}. Then the ideal $\mathcal{J}\subset
\fq[\boldsymbol{B}_0,\bfs{T}]$ generated by
$\Delta^{i-1}F(\boldsymbol{B}_{0},T_1,\ldots,T_i)$ $(1\leq i\leq r)$
is a radical ideal. Moreover, the variety $\Gamma_r^{*}$ is an
ideal-theoretic complete intersection of dimension $d-s$.
\end{corollary}
\begin{proof}
Let $J_\Delta(\bfs{B}_0,\bfs{T})$ be the Jacobian matrix  of the
polynomials \linebreak
$\Delta^{i-1}F({\boldsymbol{B}_{0},T_1,\ldots,T_i})$ ($1\leq i\leq
r$) with respect to $\bfs{B}_0,\bfs{T}$. We claim that the set of
points $(\bfs{b}_0,\bfs{\alpha}) \in\Gamma_r^*$ for which
$J_\Delta(\bfs{b}_0,\bfs{\alpha})$ has not full rank is contained in
a subvariety of $\Gamma_r^*$ of codimension 1.

Let $(\bfs{b}_0,\bfs{\alpha})$ be an arbitrary point of
$\Gamma_r^*$. In the proof of Lemma \ref{lemma: jacobian_F full rank
implies nonsingular} we show that $F(\boldsymbol{B}_0,T_j)\in
\mathcal{J}$ for $1\leq j \leq r$. This implies that the gradient
$\nabla F(\bfs{b}_0,\alpha_j)$ is a linear combination of the
gradients of the polynomials $\Delta^{i-1}F(\bfs{b}_0,\bfs{\alpha})$
for $1\leq i\leq r$. We conclude that
$\mathrm{rank}\,J_F(\bfs{b}_0,\bfs{\alpha})\le
\mathrm{rank}\,J_{\Delta}(\bfs{b}_0,\bfs{\alpha})$.

Moreover, if $J_F(\bfs{b}_0,\bfs{\alpha})$ has not full rank, then
$f_{\boldsymbol{b}_0}$ has multiple roots. By Lemma \ref{lemma:
f'=0}, the set of points $(\bfs{b}_0,\bfs{\alpha})\in\Gamma_r^*$ for
which $f'_{\boldsymbol{b}_0}=0$ is contained in a subvariety of
codimension 2 of $\Gamma_r^*$. On the other hand, if
$(\bfs{b}_0,\bfs{\alpha})$ is an arbitrary point of $\Gamma_r^*$
such that $f_{\boldsymbol{b}_0}$ has multiple roots and
$f'_{\boldsymbol{b}_0}\not=0$, then $(\bfs{b}_0,\bfs{\alpha})\in
\Psi_r^{-1}(\mathcal{Z})$, where $\mathcal{Z}$ is the subvariety of
$\A^{d-s}$ defined by the equation
$\mathrm{Res}(F(\boldsymbol{B}_0,T),
\Delta^1F(\boldsymbol{B}_0,T,T), T)=0$. We see that
$\Psi_r^{-1}(\mathcal{Z})$ has codimension 1 in $\Gamma_r^*$,
finishing thus the proof of our claim.

By Lemma \ref{lemma: Gamma r is set-theoret complete intersection}
the polynomials $\Delta^{i-1}F({\boldsymbol{B}_{0},T_1,\ldots,T_i})$
$(1\leq i\leq r)$ form a regular sequence. Therefore, by
\cite[Theorem 18.15]{Eisenbud95} we deduce that $\mathcal{J}$ is a
radical ideal, which in turn implies that $\Gamma_r^*$ is an
ideal--theoretic complete intersection of dimension $d-s$.
\end{proof}
%
%
\subsection{The geometry of the projective closure of $\Gamma_r^*$}
\label{subsec: geometry of pcl Gamma r}
In order to obtain estimates on the number of $q$-rational points of
$\Gamma_{r}^*$ we need information concerning the behavior of
$\Gamma_{r}^*$ at infinity. For this purpose, we consider the
projective closure of $\mathrm{pcl}(\Gamma_{r}^*)\subset\Pp^{d-s+r}$
of $\Gamma_{r}^*$, whose definition we now recall. Consider the
embedding of $\A^{d-s+r}$ into the projective space $\Pp^{d-s+r}$
which assigns to any point $(\boldsymbol{b}_0,\bfs{\alpha})\in
\A^{d-s+r}$ the point
$(b_{d-s-1}:\dots:b_0:1:\alpha_1:\dots:\alpha_r)\in \Pp^{d-s+r}$.
The closure $\mathrm{pcl}(\Gamma_{r}^*)\subset\Pp^{d-s+r}$ of the
image of $\Gamma_{r}^*$ under this embedding in the Zariski topology
of $\Pp^{d-s+r}$ is called the projective closure of $\Gamma_{r}^*$.
The points of $\mathrm{pcl}(\Gamma_{r}^*)$ lying in the hyperplane
$\{T_0=0\}$ are called the points of $\mathrm{pcl}(\Gamma_{r}^*)$ at
infinity.

It is well-known that $\mathrm{pcl}(\Gamma_{r}^*)$ is the
$\fq$-variety of $\Pp^{d-s+r}$ defined by the homogenization $F^h
\in \fq[\bfs{B}_0,T_0,\bfs{T}]$ of each polynomial $F$ in the ideal
$\mathcal{J}\subset \fq[\boldsymbol{B}_0,\boldsymbol{T}]$ generated
by $\Delta^{i-1}F(\boldsymbol{B}_0,T_1,\ldots,T_i)$ $(1\leq i \leq
r)$. We denote by $\mathcal{J}^{h}$ the ideal generated by all the
polynomials $F^h$ with $F\in \mathcal{J}$. Since $\mathcal{J}$ is
radical it turns out that $\mathcal{J}^h$ is also a radical ideal
(see, e.g., \cite[\S I.5, Exercise 6]{Kunz85}). Furthermore,
$\mathrm{pcl}(\Gamma_{r}^*)$ is equidimensional of dimension $d-s$
(see, e.g., \cite[Propositions I.5.17 and II.4.1]{Kunz85}) and
degree equal to $\deg\Gamma_r^*$ (see, e.g., \cite[Proposition
1.11]{CaGaHe91}).

\begin{lemma}\label{lemma: pcl Gamma r is ideal-theoret complete inters}
The homogenized polynomials
$\Delta^{i-1}F(\boldsymbol{B}_{0},T_1,\ldots,T_i)^h$ $(1\leq i \leq
r)$ generate the ideal $\mathcal{J}^h$. Furthermore,
$\mathrm{pcl}(\Gamma_{r}^*)$ is an ideal-theoretic complete
intersection of dimension $d-s$ and degree ${d!}/{(d-r)!}$.
\end{lemma}
\begin{proof}
According to Lemma \ref{lemma: Gamma r is set-theoret complete
intersection}, the polynomials
$\Delta^{i-1}F(\boldsymbol{B}_{0},T_1,\ldots,T_i)$ $(1\leq i\leq r)$
form a Gr\"obner basis of the ideal $\mathcal{J}$ with the graded
lexicographical order defined by
$T_r>\cdots>T_1>B_{d-s-1}>\cdots>B_{0}$. Therefore, the first
assertion follows from, e.g., \cite[\S 8.4, Theorem 4]{CoLiOS92}. In
particular, we have that $\mathrm{pcl}(\Gamma_{r}^*)$ is an
ideal-theoretic complete intersection. Hence, \cite[Theorem
18.3]{Harris92} proves that the degree of
$\mathrm{pcl}(\Gamma_{r}^*)$ is ${d!}/{(d-r)!}$
\end{proof}

Our next purpose is to study the singular points of
$\mathrm{pcl}(\Gamma_{r}^*)$. We start with the following
characterization of the points of $\mathrm{pcl}(\Gamma_{r}^*)$ at
infinity.
\begin{lemma}\label{lemma: singular locus pcl Gamma_r at infinity}
$\mathrm{pcl}(\Gamma_{r}^*)\cap \{T_0=0\}\subset \Pp^{d-s-1+r}$ is a
linear variety of dimension $d-s-1$.
\end{lemma}
\begin{proof}
According to Lemma \ref{lemma: pcl Gamma r is ideal-theoret complete
inters}, the homogeneous polynomials \linebreak
$\Delta^{i-1}F(\boldsymbol{B}_{0},T_1,\ldots,T_i)^h$ $(1\leq i\leq
r)$ generate the ideal $\mathcal{J}^h$. Since \linebreak
$\Delta^{i-1}F(\boldsymbol{B}_{0},T_1,\ldots,T_i)^h|_{T_0=0}=T_i^{d-i+1}+$
monomials of positive degree in $T_1\klk T_{i-1}$, we conclude that
$\mathrm{pcl}(\Gamma_{r}^*)\cap\{T_0=0\}$ is the linear
$\fq$--variety defined by the equations $\{T_1=0,\dots,T_r=0\}$.
This finishes the proof of the lemma.
\end{proof}

Now we are able to prove the main result of this section, which
summarizes all the facts we need concerning the projective variety
$\mathrm{pcl}(\Gamma_{r}^*)$.
\begin{theorem}\label{th: pcl Gamma_r is normal abs irred}
With assumptions as in Theorem \ref{th: codimension singular locus
Gamma r}, the projective variety $\mathrm{pcl}(\Gamma_{r}^*)\subset
\Pp^{d-s+r}$ is a normal absolutely irreducible ideal-theoretic
complete intersection of dimension $d-s$ and degree ${d!}/{(d-r)!}$.
\end{theorem}
\begin{proof}
Combining Theorem \ref{th: codimension singular locus Gamma r} and
Lemma \ref{lemma: singular locus pcl Gamma_r at infinity} we see
that the singular locus of $\mathrm{pcl}(\Gamma_{r}^*)$ has
codimension at least $2$ in $\mathrm{pcl}(\Gamma_{r}^*)$. This
implies that $\mathrm{pcl}(\Gamma_{r}^*)$ is regular in codimension
$1$. On the other hand, Lemma \ref{lemma: pcl Gamma r is
ideal-theoret complete inters} shows that
$\mathrm{pcl}(\Gamma_{r}^*)$ is an ideal--theoretic complete
intersection. Therefore, from the Serre criterion for normality we
deduce that $\mathrm{pcl}(\Gamma_{r}^*)$ is a normal variety.

Finally, we show that $\mathrm{pcl}(\Gamma_{r}^*)$ is absolutely
irreducible. For this purpose, we use the Hartshorne connectedness
theorem (see, e.g, \cite[\S VI.4, Theorem 4.2]{Kunz85}), which
asserts that a set-theoretic complete intersection in projective
space having a singular locus of codimension at least 2 is
absolutely irreducible. Since $\mathrm{pcl}(\Gamma_{r}^*)$ satisfies
the conditions in the statement of such a theorem, we deduce that it
is absolutely irreducible.
\end{proof}

As a consequence of Theorem \ref{th: pcl Gamma_r is normal abs
irred}, we have that $\Gamma_{r}^* \subset \A^{d-s+r}$ is also
absolutely irreducible of dimension $d-s$ and degree $d!/(d-r)!$.
Furthermore, Lemma \ref{lemma: relacion gamma_r y gamma_r estrella}
shows that $\Gamma_r$ is a nonempty open Zariski subset of
$\Gamma_r^*$. Since $\Gamma_{r}^*$ is absolutely irreducible, we
conclude that the Zariski closure $\mathrm{cl}(\Gamma_{r})$ of
$\Gamma_r$ is equal to $\Gamma_r^*$.
%
%
\section{The number of $q$-rational points of $\Gamma_{r}$}
\label{sec: number of points Gamma r}
As before, let be given positive integers $d$ and $s$ such that,
either $1\leq s\leq d-4$ and $p>3$, or $1\leq s\leq d-6$ and $p=3$.
In this section we determine the asymptotic behavior of the average
value set $\mathcal{V}(d,s,\boldsymbol{a})$ of the family of
polynomials $\{f_{\bfs{b}}:\bfs{b}\in\fq^{d-s-1}\}$.
By Theorem \ref{th: combinatorial reduction mean} we have
$$\mathcal{V}(d,s,\boldsymbol{a})= \sum_{r=1}^{d-s}(-1)^{r-1}
  \binom{q}{r}q^{1-r}
    +\frac{1}{q^{d-s-1}}\sum_{r=d-s+1}^{d}
   (-1)^{r-1}\chi_r^{\boldsymbol{a}},
$$
where $\chi_r ^{\boldsymbol{a}}$ denotes the number of subsets
$\chi_r$ of $\fq$ of $r$ elements such that there exists
$\boldsymbol{b}_0\in \fq^{d-s}$ with
$f_{\boldsymbol{b}_0}|_{\chi_r}\equiv 0$. Combining Lemmas
\ref{lemma: relacion entre gamma y chi} and \ref{lemma: relacion
gamma_r y gamma_r estrella} it follows that
$$\chi_r^{\boldsymbol{a}}=\frac{|\Gamma_{r}(\fq)|}{r!}=
\frac{1}{r!}\bigg|\Gamma_r^*(\fq)\setminus\bigcup_{i\not=j}\{T_i=T_j\}\bigg|$$
for each $r$ with $d-s+1\leq r \leq d$. In the next section we apply
the results on the geometry of $\Gamma_r^*$ of Section \ref{sec:
geometry of Gamma_r estrella} in order to obtain an estimate on the
number of $q$--rational points of $\Gamma_r^*$.
%
%
\subsection{Estimates on the number of $q$--rational points of normal complete intersections}
In what follows, we shall use an estimate on the number of
$q$--rational points of a projective normal complete intersection of
\cite{CaMaPr13} (see also \cite{CaMa07} or \cite{GhLa02a} for other
estimates). More precisely, if $V\subset \Pp^n$ is a normal complete
intersection defined over $\fq$ of dimension $m\geq 2$, degree
$\delta$ and multidegree $\boldsymbol{d}:=(d_1,\ldots,d_{n-m})$,
then the following estimate holds (see \cite[Theorem
1.3]{CaMaPr13}):
\begin{equation}\label{eq: estimate normal var CaMaPr}
\big||V(\fq)|-p_m\big| \leq
(\delta(D-2)+2)q^{m-\frac{1}{2}}+14D^2\delta^2 q^{m-1},
\end{equation}
where $p_m:=q^m+q^{m-1}+\cdots+q+1=|\Pp^m(\fq)|$ and
$D:=\sum_{i=1}^{n-m}(d_i-1)$.

From Theorem \ref{th: pcl Gamma_r is normal abs irred} we have that
the projective variety
$\mathrm{pcl}(\Gamma_{r}^*)\subset\Pp^{d-s+r}$ is a normal complete
intersection defined over $\fq$. Therefore, applying (\ref{eq:
estimate normal var CaMaPr}) we obtain:
$$
\big||\mathrm{pcl}(\Gamma_{r}^*)(\fq)|-p_{d-s}\big| \leq
(\delta_r(D_r-2)+2)q^{d-s-\frac{1}{2}}+14D_r^2\delta_r^2q^{d-s-1},
$$
where $D_r:=\sum_{i=1}^r(d-i)=rd-{r(r+1)}/{2}$ and
$\delta_r:={d!}/{(d-r)!}$. On the other hand, since
$\mathrm{pcl}(\Gamma_{r}^*)^{\infty}:=\mathrm{pcl}(\Gamma_{r}^*)\cap
\{T_0=0\}\subset\Pp^{d-s-1+r}$ is a linear variety of dimension
$d-s-1$, the number of $q$--rational points of
$\mathrm{pcl}(\Gamma_{r}^*)^{\infty}$ is $p_{d-s-1}$. Hence we have:
\begin{eqnarray}
\big||\Gamma_{r}^*(\fq)|-q^{d-s}\big|&=&
\big||\mathrm{pcl}(\Gamma_{r}^*)(\fq)|-|\mathrm{pcl}(\Gamma_{r}^*(\fq))^{\infty}|-p_{d-s}+p_{d-s-1}\big|\nonumber\\
&=& \big||\mathrm{pcl}(\Gamma_{r}^*)(\fq)|-p_{d-s}\big| \nonumber\\
&\le&
(\delta_r(D_r-2)+2)q^{d-s-\frac{1}{2}}+14D_r^2\delta_r^2q^{d-s-1}.
\label{eq: estimate q-points Gamma_r estrella}
\end{eqnarray}

We also need an estimate on the number $q$--rational points of the
affine $\fq$--variety
$$
\Gamma_{r}^{*,=}:= \Gamma_{r}^* \bigcap \bigcup_{1 \leq i <j \leq r}
\{T_i =T_j\}.
$$
We observe that $\Gamma_{r}^{*,=}=\Gamma_{r}^{*}\cap \mathcal{H}_r$,
where $\mathcal{H}_r\subset \A^{d-s+r}$ is the hypersurface defined
by the polynomial $F_r:=\prod_{1 \leq i <j \leq r} (T_i-T_j)$. From
the B\'ezout inequality (\ref{eq: Bezout inequality}) it follows
that
\begin{equation}\label{eq: deg Gamma =}
\deg \Gamma_{r}^{*,=} \leq \delta_r \binom {r}{2},
\end{equation}
Furthermore, we claim that $\Gamma_{r}^{*,=}$ has dimension at most
$d-s-1$. Indeed, let $(\bfs{b}_0,\bfs{\alpha})$ be an arbitrary
point of $\Gamma_r^{*,=}$. Assume without loss of generality that
$\alpha_1=\alpha_2$. By the definition of the divided differences we
deduce that $f'_{\boldsymbol{b}_0}(\alpha_1)=0$, which implies that
$f_{\boldsymbol{b}_0}$ has multiple roots. In Corollary \ref{coro: J
is radical} we prove that the set of points of $\Gamma_r^*$ for
which $f_{\boldsymbol{b}_0}$ has multiple roots is contained in a
subvariety of $\Gamma_r^*$ of codimension at least 1. Therefore, we
deduce our claim.

Combining our claim with (\ref{eq: deg Gamma =}), applying, e.g.,
\cite[Lemma 2.1]{CaMa06}, we obtain
\begin{equation}\label{eq: upper bound q-points gamma_r=}
\big|\Gamma_{r}^{*,=}(\fq)\big| \leq \delta_r\binom{r}{2} q^{d-s-1}.
\end{equation}
Since $\Gamma_{r}(\fq)=\Gamma_{r}^*(\fq)\setminus
\Gamma_{r}^{*,=}(\fq)$, from (\ref{eq: estimate q-points Gamma_r
estrella}) and (\ref{eq: upper bound q-points gamma_r=}) we deduce
that
\begin{eqnarray*}
\big||\Gamma_{r}(\fq)|- q^{d-s}\big|\!\!\!\!&\le\!\!\!\!&
\big||\Gamma_{r}^*(\fq)|- q^{d-s}\big|+
|\Gamma_{r}^{*,=}(\fq)|\\
\!\!\!\!&\le\!\!\!\!&(\delta_r(D_r-2)+2)q^{d-s-\frac{1}{2}}\!+\!\big(14D_r^2\delta_r^2+{r(r-1)\delta_r}/{2}
\big)q^{d-s-1}.
\end{eqnarray*}
As a consequence, we obtain the following result.
\begin{theorem}\label{th: estimate chi_r} If either $1\leq s \leq d-3$
and $p>3$, or $1\leq s\leq d-6$ and $p=3$, then for any $r$ with
$d-s+1\leq r\leq d$ we have
$$
\bigg|\mathcal{\chi}_{r}^{\boldsymbol{a}}-
\frac{q^{d-s}}{r!}\bigg|\leq
\frac{1}{r!}(\delta_r(D_r-2)+2)q^{d-s-\frac{1}{2}}+
\frac{1}{r!}\big(14D_r^2\delta_r^2+r(r-1)\delta_r/{2}\big)q^{d-s-1},
$$
where $D_r:=rd-{r(r+1)}/{2}$ and $\delta_r:={d!}/{(d-r)!}$.
\end{theorem}
%
%
\subsection{An estimate for the average mean value $\mathcal{V}(d,s,\bfs{a})$}
Theorem \ref{th: estimate chi_r} is the critical step in our
approach to estimate the average mean value
$\mathcal{V}(d,s,\boldsymbol{a})$.
\begin{corollary}\label{coro: average value sets}
With assumptions and notations as in Theorem \ref{th: estimate
chi_r}, we have
\begin{equation}\label{eq: estimate V(d,s,a)}
\left|\mathcal{V}(d,s,\bfs{a})-\mu_d\, q\right|\le d^22^{d-1}
q^{1/2}+ \frac{7}{2}\,d^4\sum_{k=0}^{s-1}\binom{d}{k}^{2}(d-k)!.
\end{equation}
%
\end{corollary}
\begin{proof}
According to Theorem \ref{th: combinatorial reduction mean}, we have
\begin{equation}\label{eq: average value set th 2}
\mathcal{V}(d,s,\bfs{a})-\mu_d\, q=
\sum_{r=1}^{d-s}(-q)^{1-r}\!\!\left(\!\binom {q} {r}
-\frac{q^r}{r!}\!\right)
+\frac{1}{q^{d-s-1}}\hskip-0.25cm\sum_{r=d-s+1}^{d} \hskip-0.25cm
(-1)^{r-1}\!\!\left(\chi_r^{\bfs{a}}-\frac{q^{d-s}}{r!}\!\right).
\end{equation}

In \cite[Corollary 14]{CeMaPePr13} we obtain the following upper
bound for the absolute value of the first term in the right--hand
side of (\ref{eq: average value set th 2}):
$$A(d,s):=\bigg|\sum_{r=1}^{d-s}(-q)^{1-r}\left(\binom {q} {r}
-\frac{q^r}{r!}\right)\bigg|\le
\frac{1}{2\cdot(d-s-1)!}+\frac{7}{q}+\frac{1}{2e}\le d.
$$

Next we consider the absolute value of the second term in the
right--hand side of (\ref{eq: average value set th 2}). From Theorem
\ref{th: estimate chi_r} we have that
\begin{align*}B(d,s)&:=\frac{1}{q^{d-s-1}}\sum_{r=d-s+1}^{d}
\left|\chi_r^{\bfs{a}}-\frac{q^{d-s}}{r!}\right|\\
&\le q^{1/2}\sum_{r=d-s+1}^{d}\frac{\delta_r(D_r-2)+2}{r!}\,+
14\sum_{r=d-s+1}^{d}\frac{D_r^{2}\delta_r^{2}}{r!}+
\sum_{r=d-s+1}^{d}\frac{\delta_r}{2(r-2)!}.
\end{align*}
Concerning the first term in the right--hand side, we see that
$$
\sum_{r=d-s+1}^{d}\frac{\delta_r(D_r-2)+2}{r!}\leq
\sum_{r=d-s+1}^{d}\binom{d}{r}\frac{r(2d-1-r)}{2}\leq d^22^{d-1}.
$$
On the other hand,
$$\sum_{r=d-s+1}^{d}\!\!\frac{D_r^2\delta_r^{2}}{r!}\!=
\!\!\sum_{r=d-s+1}^{d}\!\!\binom{d}{r}^{2}\frac{r^2(2d-1-r)^2\,r!}{4}
\leq\frac{1}{64}(2d-1)^{4}\sum_{k=0}^{s-1}\binom{d}{k}^{2}(d-k)!.$$
Finally, we consider the last sum
$$
\sum_{r=d-s+1}^{d}\frac{\delta_r}{2(r-2)!}=
\sum_{r=d-s+1}^d\binom{d}{r}\frac{r(r-1)}{2}=\sum_{k=0}^{s-1}\binom{d}{k}\frac{(d-k)!}{2\,(d-k-2)!}.
$$
Therefore, we obtain
$$B(d,s)\le q^{1/2}d^22^{d-1}+\frac{1}{4}\sum_{k=0}^{s-1}\binom{d}{k}(d-k)!+
\frac{7}{32}(2d-1)^4\sum_{k=0}^{s-1}\binom{d}{k}^{2}(d-k)!.$$

Combining the bounds for $A(d,s)$ and $B(d,s)$ the statement of the
corollary follows.
\end{proof}
%
%
\subsection{On the behavior of (\ref{eq: estimate V(d,s,a)})}
\label{subsec: behavior V(d,s,a)}
In this section we analyze the behavior of the right--hand side of
(\ref{eq: estimate V(d,s,a)}). Such an analysis consists of
elementary calculations, which shall only be sketched.

Fix $k$ with $0\le k\le s-1$ and denote
$h(k):=\binom{d}{k}^2(d-k)!$. Analyzing the sign of the differences
$h(k+1)-h(k)$ for $0\le k\le s-1$, we deduce the following remark,
which is stated without proof.
\begin{remark}\label{rem: growth h(k)}
Let $k_0:=-1/2+\sqrt{5+4d}/2$. Then $h$ is either an increasing
function or a unimodal function in the integer interval $[0,s-1]$,
which reaches its maximum at $\lfloor k_0\rfloor$.
\end{remark}

From Remark \ref{rem: growth h(k)} we see that
\begin{equation}\label{eq: expresion a analizar}
\sum_{k=0}^{s-1}\binom{d}{k}^{2}(d-k)!\le s \binom{d}{\lfloor
k_0\rfloor}^2(d-\lfloor k_0\rfloor)!=\frac{s\,(d!)^2}{(d-\lfloor
k_0\rfloor)!\,(\lfloor k_0\rfloor!)^2 }.
\end{equation}
In order to obtain an upper bound for the right--hand side of
(\ref{eq: expresion a analizar}) we shall use the Stirling formula
(see, e.g., \cite[p. 747]{FlSe08}): for $m\in \N$, there exists
$\theta$ with $0\le \theta<1$ such that $m!=(m/e)^m\sqrt{2\pi
m}\,e^{\theta/12m}$ holds.

Applying the Stirling formula, we see that there exist $\theta_i$
($i=1,2,3$) with $0\le \theta_i<1$ such that
$$C(d,\!s)\!:=\frac{s\,(d!)^2}{(d-\lfloor
k_0\rfloor)!\,(\lfloor k_0\rfloor!)^2 }\le \frac{s\,
d^{2d+1}e^{-d+\lfloor
k_0\rfloor}e^{\frac{\theta_1}{6d}-\frac{\theta_2}{12(d-\lfloor
k_0\rfloor)}-\frac{\theta_3}{6\lfloor k_0\rfloor}}}{\big(d-\lfloor
k_0\rfloor\big)^{d-\lfloor k_0\rfloor}\sqrt{2\pi(d-\lfloor
k_0\rfloor)}\lfloor k_0\rfloor^{2\lfloor k_0\rfloor+1}}.$$
By elementary calculations we obtain
\begin{eqnarray*}
 (d-\lfloor k_0\rfloor)^{-d+\lfloor k_0\rfloor}&\le &
 d^{-d+\lfloor k_0\rfloor}e^{{\lfloor k_0\rfloor(d-\lfloor k_0\rfloor)}/{d}}, \\
 \frac{ d^{\lfloor k_0\rfloor}}{{\lfloor k_0\rfloor}^{2\lfloor k_0\rfloor}} &\le &
 e^{(d-\lfloor k_0\rfloor^2)/\lfloor k_0\rfloor}.
\end{eqnarray*}
It follows that
$$C(d,s)\le \frac{s \,d^{d+1}e^{2 \lfloor
k_0\rfloor}e^{-\frac{\lfloor k_0\rfloor^2}{d} +
\frac{1}{6d}+\frac{d-\lfloor k_0 \rfloor ^2}{\lfloor
k_0\rfloor}}}{\sqrt{2\pi} e^d \sqrt{d-\lfloor k_0\rfloor}\lfloor
k_0\rfloor}.
$$

By the definition of $\lfloor k_0\rfloor$, it is easy to see that
${d}/{\lfloor k_0\rfloor\sqrt{d-\lfloor k_0\rfloor}}\le {11}/{6}$
and that $2\lfloor{k_0}\rfloor \le -1+\sqrt{5+4d}\le -1/5+2\sqrt{d}$.
Therefore, taking into account that $d\ge 5$, we conclude that
$$
C(d,s)\le \frac{11}{6}\,\frac{e^{\frac{487}{165}}s\,
d^d\,e^{2\sqrt{d}}}{ \sqrt{2\pi }e^d}.
$$
Combining this bound with Corollary \ref{coro: average value sets}
we obtain the following result.
\begin{theorem}\label{theorem: final main result - mean}
With assumptions and notations as in Theorem \ref{th: estimate
chi_r}, we have
$$
\left|\mathcal{V}(d,s,\bfs{a})-\mu_d\,q\right|\le d^2
2^{d-1}q^{1/2} +49 \,{d^{d+5} e^{2 \sqrt{d}-d}}. $$
\end{theorem}
%
%
%
%
\section{Estimating the second moment
$\mathcal{V}_2(d,s,\boldsymbol{a})$: combinatorial preliminaries}
\label{sec: combinatorial preliminaries variance}
Now we consider the second objective of this paper: estimating the
second moment of the value set of the families of polynomials under
consideration.

As before, we assume that the characteristic $p$ of $\fq$ is greater
than 2 and fix integers $d$ and $s$ with $d<q$, $1 \leq s \leq d-4$
for $p>3$, and $1\le s\le d-6$ for $p=3$. We also fix
$\boldsymbol{a} :=(a_{d-1},\dots,a_{d-s}) \in \mathbb{F}_q^s$ and
set $f_{\boldsymbol{a}}:= T^d+a_{d-1}T^{d-1}+\cdots+a_{d-s}T^{d-s}$.
Further, for any $\boldsymbol{b}:=(b_{d-s-1},\dots,b_1)\in
\mathbb{F}_q^{d-s-1}$, we denote
$$f_{\boldsymbol{b}}:=
T^d+a_{d-1}T^{d-1}+\cdots+a_{d-s}T^{d-s}+b_{d-s-1}T^{d-s-1}+\cdots+b_1T.$$
In what follows, we shall consider the problem of estimating the
following sum:
\begin{equation}\label{eq: formula para v2}
\mathcal{V}_2(d,s,\boldsymbol{a}):=\dfrac{1}{q^{d-s-1}}\sum_{\boldsymbol{b}
\in \mathbb{F}_q^{d-s-1}}\mathcal{V}(f_{\boldsymbol{b}})^2.
\end{equation}
We start with the following result, which plays a similar role as
Theorem \ref{th: combinatorial reduction mean} in the estimate of
$\mathcal{V}(d,s,\bfs{a})$.
\begin{theorem}\label{th: combinatorial reduction variance}
Let assumptions be as above. We have
$$\begin{array}{rcl}
  \mathcal{V}_2(d,s,\boldsymbol{a})=
  \mathcal{V}(d,s,\boldsymbol{a})&+&
  \displaystyle\mathop{\sum_{1\leq m,n\leq d}}_{2\leq m+n\leq d-s}(-1)^{m+n}
  \binom{q}{m}\binom{q}{n}q^{2-n-m}\\
  \\
   & +&\dfrac{1}{q^{d-s-1}}\hskip-0.5cm\displaystyle\mathop{\sum_{1\leq m,n\leq d}}_{d-s+1\leq m+n\leq 2d}
 \hskip-0.5cm(-1)^{m+n}\hskip-0.5cm\displaystyle\mathop{\sum_{\Gamma_1,\Gamma_2 \subset\fq}}_{|\Gamma_1|=m,|\Gamma_2|=n}
  \left|S_{\Gamma_1,\Gamma_2}^{\boldsymbol{a}}\right|,
\end{array}
$$
where $\mathcal{S}_{\Gamma_1,\Gamma_2}^{\boldsymbol{a}}$ is the set
consisting of the points $(\boldsymbol{b},b_{0,1},b_{0,2})\in
\mathbb{F}_q^{d-s+1}$ with $b_{0,1}\neq b_{0,2}$ such that
$(f_{\boldsymbol{b}}+b_{0,1})\big|_{\Gamma_1}\equiv 0$ and
$(f_{\boldsymbol{b}}+b_{0,2}){\big|_{\Gamma_2}}\equiv 0$ holds.
\end{theorem}
%
%
\begin{proof}
Fix $\boldsymbol{b}\in \mathbb{F}_q^{d-s-1}$. Let $\fq[T]_d$ denote
the set of polynomials of $\fq[T]$ of degree at most $d$, let
$\mathcal{N}:\fq[T]_d\to\Z_{\ge 0}$ be the counting function of the
number of roots in $\fq$ and let
$\bfs{1}_{\{\mathcal{N}>0\}}:\fq[T]_d\to\{0,1\}$ be the
characteristic function of the set of elements of $\fq[T]_d$ having
at least one root in $\fq$. Taking into account that
$\mathcal{V}(f_{\boldsymbol{b}})= \sum_{b_0\in
\mathbb{F}_q}\textbf{1}_{\{\mathcal{N}>0\}}(f_{\boldsymbol{b}}+b_0)$,
we obtain
$$q^{d-s-1}\mathcal{V}_2(d,s,\boldsymbol{a})=\!\!\!\!\sum_{\boldsymbol{b}\in
\mathbb{F}_q^{d-s-1}}\!\!\!\Bigg(\sum_{b_{0,1} \in \mathbb{F}_q}
\!\!\textbf{1}_{\{\mathcal{N}>0\}}(f_{\boldsymbol{b}}+b_{0,1})\Bigg)\Bigg(\sum_{b_{0,2}
\in \mathbb{F}_q}\!\!
\textbf{1}_{\{\mathcal{N}>0\}}(f_{\boldsymbol{b}}+b_{0,2})\Bigg).$$
For a given $(\boldsymbol{b},b_{0,1},b_{0,2})\in
\mathbb{F}_q^{d-s+1}$, we denote
$f_{\bfs{b}_1}:=f_{\bfs{b}}+b_{0,1}$ and
$f_{\bfs{b}_2}:=f_{\bfs{b}}+b_{0,2}$. We have
\begin{align*}
q^{d-s-1}\mathcal{V}_2(d,s,\boldsymbol{a})=\sum_{\boldsymbol{b}\in
\mathbb{F}_q^{d-s-1}}\sum_{(b_{0,1},b_{0,2})\in \mathbb{F}_q^2
}\!\!\!\!\textbf{1}_{\{\mathcal{N}>0\}^2}(f_{\boldsymbol{b}_1},f_{\boldsymbol{b}_2})\notag
\hskip3.5cm\\
=\sum_{\boldsymbol{b}\in
\mathbb{F}_q^{d-s-1}}\mathop{\sum_{(b_{0,1},b_{0,2})\in
\mathbb{F}_q^2}}_{b_{0,1}=b_{0,2}
}\!\!\!\!\textbf{1}_{\{\mathcal{N}>0\}^2}(f_{\boldsymbol{b}_1},f_{\boldsymbol{b}_2})
 +\!\!\!\!\sum_{\boldsymbol{b}\in
\mathbb{F}_q^{d-s-1}}\mathop{\sum_{(b_{0,1},b_{0,2})\in
\mathbb{F}_q^2}}_{ b_{0,1}\neq b_{0,2}
}\!\!\!\!\textbf{1}_{\{\mathcal{N}>0\}^2}(f_{\boldsymbol{b}_1},f_{\boldsymbol{b}_2}).
\end{align*}

Concerning the first term in the right--hand side of the last
equality, we have
\begin{equation}\label{eq: primera parte de la suma teo varianza}
  \mathop{\sum_{(\boldsymbol{b},b_{0,1},b_{0,2})\in \fq^{d-s+1}}}_{ b_{0,1}=b_{0,2}}\textbf{1}_{\{\mathcal{N}>0\}^2}
  (f_{\boldsymbol{b}_1},f_{\boldsymbol{b}_2})=
  \sum_{\boldsymbol{b}_1\in \fq^{d-s}}\textbf{1}_{\{\mathcal{N}>0\}}
  (f_{\boldsymbol{b}_1})= q^{d-s-1}\mathcal{V}(d,s,\boldsymbol{a}).
  \end{equation}

Next we analyze the second term of the expression for
$\mathcal{V}_2(d,s,\bfs{a})$ under consideration. For this purpose,
we express it in terms of cardinality of the sets
$$S_{\{\alpha\},\{\beta\}}^{\boldsymbol{a}}:=
\left\{(\boldsymbol{b},b_{0,1},b_{0,2})\in \mathbb{F}_q^{d-s+1}: \,
b_{0,1} \neq b_{0,2}, \, \, f_{\boldsymbol{b}_1}(\alpha)=
f_{\boldsymbol{b}_2}(\beta)=0\right\}$$
with $\alpha,\beta\in\fq$. More precisely, we have
\begin{eqnarray*}
\mathop{\sum_{(\boldsymbol{b},b_{0,1},b_{0,2})\in \fq^{d-s+1}}}_{
b_{0,1}\neq b_{0,2}}\!\!
\textbf{1}_{\{\mathcal{N}>0\}^{2}}{(f_{\boldsymbol{b}_1},f_{\boldsymbol{b}_2})}
&=&\Bigg|\mathop{\bigcup_{\{\alpha,\beta\} \subseteq \fq}}_{
\alpha\neq\beta}
\mathcal{S}_{\{\alpha\},\{\beta\}}^{\boldsymbol{a}}\Bigg|=\Bigg|\bigcup_{\alpha
\in \mathbb{F}_q }\mathop{\bigcup_{\beta \in
\mathbb{F}_q}}_{\alpha\neq\beta}
\mathcal{S}_{\{\alpha\},\{\beta\}}^{\boldsymbol{a}}\Bigg|.
\end{eqnarray*}
Let
$\mathcal{T}_{\alpha}^{\bfs{a}}:=\bigcup_{\beta\in\fq}\mathcal{S}_{\{\alpha\},\{\beta\}}^{\bfs{a}}$.
By the inclusion-exclusion principle we obtain
\begin{eqnarray*}
\Bigg|\bigcup_{\alpha \in \mathbb{F}_q }\bigcup _{\beta \in
\mathbb{F}_q}
\mathcal{S}_{\{\alpha\},\{\beta\}}^{\boldsymbol{a}}\Bigg|
&=& \sum_{m=1}^q (-1)^{m-1}\sum_{\{\alpha_1,\ldots, \alpha_m\}\subset\fq}
\big| \mathcal{T}_{\alpha_1}^{\boldsymbol{a}}\cap\cdots\cap T_{\alpha_m}^{\boldsymbol{a}}\big|\\
&=&\sum_{m=1}^q (-1)^{m-1}\sum_{\{\alpha_1,\ldots,\alpha_m\}\subset\fq}
\Bigg|\bigcup_{\beta \in \mathbb{F}_q} \mathcal{S}_{\{\alpha_1,\ldots,\alpha_m\},\{\beta\}}^{\boldsymbol{a}}\Bigg|\\
&=&\mathop{\sum_{m=1}^q}_{n=1}(-1)^{m+n}\mathop{\sum_{\{\alpha_1,\ldots,\alpha_m\}\subset\fq}}_{\{\beta_1,\ldots,\beta_n\}\subset\fq}
\big|\mathcal{S}_{\{\alpha_1,\ldots,\alpha_m\},\{\beta_1,\ldots,\beta_n\}}^{\boldsymbol{a}}\big|\\
&=&\mathop{\sum_{m=1}^q}_{n=1}(-1)^{m+n}\mathop{\mathop{\sum_{\Gamma_1,\Gamma_2\subseteq\mathbb{F}_q}}_{|\Gamma_1|=m,\,|\Gamma_2|=n}}
\big|S_{\Gamma_1,\Gamma_2}^{\boldsymbol{a}}\big|.
\end{eqnarray*}
Observe that if $\Gamma_1 \cap \Gamma _2 \neq \emptyset $, then
$\mathcal{S}_{\Gamma_1,\Gamma_2}^{\boldsymbol{a}}= \emptyset$, and
that, if $m>d$ or $n>d$, then
$\mathcal{S}_{\Gamma_1,\Gamma_2}^{\boldsymbol{a}}= \emptyset$. Thus
we conclude that
$$
\mathop{\sum_{(\boldsymbol{b},b_{0,1},b_{0,2})\in
\mathbb{F}_q^{d-s+1}}}_{ b_{0,1}\neq
b_{0,2}}\textbf{1}_{\{\mathcal{N}>0\}^{2}
}{(f_{\boldsymbol{b}_1},f_{\boldsymbol{b}_2})}
                =\sum_{1\leq m,\,n\leq d}(-1)^{m+n}
\mathop{\mathop{\sum_{\Gamma_1,\Gamma_2\subseteq
\mathbb{F}_q}}_{|\Gamma _1|=m, |\Gamma_2|=n}}_{ \Gamma_1 \cap
\Gamma_2
=\emptyset}\big|\mathcal{S}_{\Gamma_1,\Gamma_2}^{\boldsymbol{a}}\big|.
$$

Fix $n,\,m\in \mathbb{N}$ and subsets
$\Gamma_1=\{\alpha_1,\ldots,\alpha_m\} \subset\fq$  and
$\Gamma_2=\{\beta_1,\ldots\beta_n\} \subset \fq$ with $\Gamma_1 \cap
\Gamma_2 =\emptyset$. If $(\boldsymbol{b},b_{0,1},b_{0,2}) \in
\mathcal{S}_{\Gamma_1,\Gamma_2}^{\boldsymbol{a}}$, then $b_{0,1}
\neq b_{0,2}$, $f_{\boldsymbol{b}_1}\big|_{\Gamma_1}\equiv 0$ and
$f_{\boldsymbol{b}_2}\big|_{\Gamma_2}\equiv 0$. These two identities
can be expressed in matrix form as follows:
\begin{equation}\label{eq: matrix id - variance}
M(\Gamma_1,\Gamma_2)\cdot
\boldsymbol{{v}}=-f_{\boldsymbol{a}}(\Gamma_1, \Gamma_2)
\end{equation}
where
$\boldsymbol{{v}}^t:=(\boldsymbol{b},b_{0,1},b_{0,2})\in\fq^{d-s+1}$
and $M(\Gamma_1,\Gamma_2)\in \mathbb{F}_q^{(m+n)\times (d-s+1)}$ and
$f_{\boldsymbol{a}}(\Gamma_1,\Gamma_2)\in\fq^{(m+n)\times 1}$ are
the following matrices:
$$M(\Gamma_1,\Gamma_2)=:\left(\begin{array}{ccccc}
\alpha_1^{d-s-1}&\cdots &\alpha_1 &1 &0 \\
\vdots& &\vdots &\vdots & \vdots\\
\alpha_m^{d-s-1}&\cdots & \alpha_m& 1&0 \\
\beta_1^{d-s-1}&\cdots &\beta_1 &0 &1 \\
\vdots& &\vdots &\vdots &\vdots \\
\beta_n^{d-s-1}&\cdots &\beta_n &0 &1
\end{array}\right),\ f_{\boldsymbol{a}}(\Gamma_1,\Gamma_2):=\left(
\begin{array}{c}
-f_{\boldsymbol{a}}(\alpha_1)\\\vdots\\-f_{\boldsymbol{a}}(\alpha_m)\\
-f_{\boldsymbol{a}}(\beta_1)\\\vdots\\-f_{\boldsymbol{a}}(\beta_n)
\end{array}\right).$$
It follows that $(\boldsymbol{b},b_{0,1},b_{0,2})\in
\mathcal{S}_{\Gamma_1,\Gamma_2}^{\boldsymbol{a}}$ if and only if
$(\boldsymbol{b},b_{0,1},b_{0,2})$  is a solution of \eqref{eq:
matrix id - variance}.

For $m+n < d-s+1$, the rank of the matrix $M(\Gamma_1,\Gamma_2)$ is
$m+n$, and the set of solutions
$\mathcal{S}_{\Gamma_1,\,\Gamma_2}^{\boldsymbol{a}}$ is a linear
$\mathbb{F}_q$--variety of dimension $d-s+1-m-n$. From (\ref{eq:
matrix id - variance}) we conclude that
$$|\mathcal{S}_{\Gamma_1,\,\Gamma_2}^{\boldsymbol{a}}|=q^{d-s+1-m-n}.$$
This implies
\begin{eqnarray*}
\mathop{\sum_{(\boldsymbol{b},b_{0,1},b_{0,2})\in
\mathbb{F}_q^{d-s+1}}}_{ b_{0,1}\neq
b_{0,2}}\textbf{1}_{\{\mathcal{N}>0\}^2}{(f_{\boldsymbol{b}_1},f_{\boldsymbol{b}_2})}
  &=&\hskip-0.5cm\displaystyle\mathop{\sum_{1\leq m\,,n\leq d}}_{2\leq m+n\leq d-s}\hskip-0.5cm(-1)^{m+n}q^{d-s+1-m-n}  \binom{q}{m}\binom{q}{n}
 \\&&+\hskip-0.5cm
 \mathop{\sum_{1\leq m\,,n\leq d}}_{d-s+1\leq m+n\leq 2d}\hskip-0.5cm(-1)^{m+n}\hskip-0.5cm
 \mathop{\mathop{\sum_{\Gamma_1,\,\Gamma_2\subseteq \mathbb{F}_q}}_{|\Gamma _1|=m, |\Gamma_2|=n}}_{ \Gamma_1 \cap \Gamma_2 =\emptyset}
 \big|\mathcal{S}_{\Gamma_1,\,\Gamma_2}^{\boldsymbol{a}}\big|.
\end{eqnarray*}
Combining \eqref{eq: primera parte de la suma teo varianza} with the
previous equality we deduce the statement of the theorem.
\end{proof}

Fix $s$, $d$ and $\bfs{a}$ as in the statement of Theorem \ref{th:
combinatorial reduction variance}. According to Theorem \ref{th:
combinatorial reduction variance}, in order to obtain a suitable
estimate for $\mathcal{V}_2(d,s,\bfs{a})$ we have to estimate the
sum
\begin{equation}\label{eq: S mn como suma sobre los gammas}
\mathcal{S}_{m,n}^{\bfs{a}}:=\mathop{\sum_{\Gamma_1,\,\Gamma_2\subset
\fq}}_{|\Gamma_1|=m , |\Gamma_2|=n}|\mathcal{S}_{\Gamma_1,\,
\Gamma_2}^{\boldsymbol{a}}|
\end{equation}
for each pair $(m,n)$ with $1\le m,n\le d$ and $d-s+1\le m+n\le 2d$.
%
%
\section{A geometric approach to estimate $\mathcal{S}_{m,n}^{\bfs{a}}$}
\label{sec: geometric approach variance}
Fix $m$ and $n$ with $1\le m,n\le d$ and $d-s+1\le m+n\le 2d$. In
order to find an estimate for  $\mathcal{S}_{m,n}^{\bfs{a}}$ we
introduce new indeterminates $T,T_1,\ldots,T_m$, $U,U_1,\ldots,U_n$,
$B,B_{d-s-1},\ldots,B_1$, $B_{0,1}$, $B_{0,2}$ over $\cfq$ and
denote $\boldsymbol{T}:=(T_1,\ldots,T_m)$,
$\boldsymbol{U}:=(U_1,\ldots,U_n)$,
$\boldsymbol{B}:=(B_{d-s-1},\ldots,B_{1})$,
$\boldsymbol{B}_1:=(\bfs{B},B_{0,1})$ and
$\boldsymbol{B}_2:=(\bfs{B},B_{0,2})$. Furthermore, we consider the
polynomial $F\in \fq[\boldsymbol{B},B,T]$ defined as follows:
\begin{equation}\label{eq: definition gB}
F:=T^d+ \sum_{i=d-s}^{d-1}a_i T_j^i + \sum_{i=1}^{d-s-1}B_i T^i +B,
\end{equation}

Observe that, for any
$(\bfs{b},b_{0,1},b_{0,2},\bfs{\alpha},\bfs{\beta})\in\fq^{d-s+1+m+n}$,
we have that $F(\bfs{b},b_{0,1},\alpha_j)=f_{\bfs{b}_1}(\alpha_j)$
and $F(\bfs{b},b_{0,2},\beta_k)=f_{\bfs{b}_2}(\beta_k)$ for $1\le
j\le m$ and $1\le k\le n$. Let $\Gamma_{m,n}\subset\A^{d-s+1+m+n}$
be the affine quasi--$\fq$--variety defined as
\begin{eqnarray*}
\Gamma_{m,n}:=\{(\boldsymbol{b},b_{0,1},b_{0,2},
\boldsymbol{\alpha},\boldsymbol{\beta})
\in\A^{d-s+1+m+n}\!:F({\boldsymbol{b},b_{0,1},\alpha_j})=0\, (1\leq
j\leq m),\hskip0.25cm\\\alpha_i\neq\alpha_j\, (i\not=\!j),\,
F({\boldsymbol{b},b_{0,2},\beta_k})=0\ (1\leq k\leq n),\
\beta_i\neq\beta_j\ (i\not=j),\ b_{0,1}\neq b_{0,2}\}.
\end{eqnarray*}
Similarly to Lemma \ref{lemma: relacion entre gamma y chi}, we have
the following result.
\begin{lemma} \label{lemma: relacion entre gamma m,n y S m,n}
Let $m$ and $n$ be integers with $1\le m,n\le d$ and $d-s+1\le
m+n\le 2d$. Then the following identity holds:
$$\frac{|\Gamma_{m,n}(\fq)|}{m!\,n!}=\mathcal{S}_{m,n}^{\bfs{a}}.$$
\end{lemma}
\begin{proof}
Let $(\boldsymbol{b},b_{0,1},b_{0,2},
\boldsymbol{\alpha},\boldsymbol{\beta})$ be an arbitrary point of
$\Gamma_{m,n}(\fq)$ and let $\sigma:\{1,\dots,m\}\to\{1,\dots,m\}$
and $\tau:\{1,\dots,n\}\to\{1,\dots,n\}$ be two arbitrary
permutations. Let $\sigma(\boldsymbol{\alpha})$ and
$\tau(\bfs{\beta})$ be the images of $\bfs{\alpha}$ and
$\bfs{\beta}$ by the linear mappings induced by these permutations.
Then it is clear that $\big(\boldsymbol{b},b_{0,1},b_{0,2},
\sigma(\boldsymbol{\alpha}),\tau(\boldsymbol{\beta})\big)$ belongs
to $\Gamma_{m,n}(\fq)$. Furthermore,
$\sigma(\boldsymbol{\alpha})=\bfs{\alpha}$ if and only if $\sigma$
is the identity permutation and a similar remark can be made
concerning $\tau(\boldsymbol{\beta})$. This shows that the product
$\mathbb{S}_m\times \mathbb{S}_n$ of the symmetric groups
$\mathbb{S}_m$ and $\mathbb{S}_n$ of $m$ and $n$ elements acts over
the set $\Gamma_{m,n}(\fq)$ and each orbit under this action has
$m!n!$ elements.

The orbit of an arbitrary point $(\boldsymbol{b},b_{0,1},b_{0,2},
\bfs{\alpha},\bfs{\beta})$ uniquely determines polynomials
$f_{\bfs{b}_1}$ and $f_{\bfs{b}_2}$ and sets
$\Gamma_1:=\{\alpha_1,\dots,\alpha_m\}\subset\fq$ and
$\Gamma_2:=\{\beta_1\klk \beta_n\}\subset\fq$ with $|\Gamma_1|=m$
and $|\Gamma_2|=n$ such that $f_{\bfs{b}_1}|_{\Gamma_1}\equiv 0$ and
$f_{\bfs{b}_2}|_{\Gamma_2}\equiv 0$ hold. Therefore, each orbit
uniquely determines sets $\Gamma_1,\Gamma_2\subset\fq$ with
$|\Gamma_1|=m$ and $|\Gamma_2|=n$ and an element of $S_{\Gamma_1,\,
\Gamma_2}^{\boldsymbol{a}}$. Reciprocally, to each element of
$S_{\Gamma_1,\,\Gamma_2}^{\boldsymbol{a}}$ there corresponds a
unique orbit of $\Gamma_{m,n}(\fq)$. This implies that
$$\mbox{number of
orbits of
}\Gamma_{m,n}(\fq)=\mathop{\sum_{\Gamma_1,\,\Gamma_2\subset
\fq}}_{|\Gamma_1|=m , |\Gamma_2|=n}\big|S_{\Gamma_1,\,
\Gamma_2}^{\boldsymbol{a}}\big|$$
and finishes the proof of the lemma.
\end{proof}

In order to estimate the quantity $|\Gamma_{m,n}(\fq)|$ we shall
consider the Zariski closure $\mathrm{cl}(\Gamma_{m,n})$ of
$\Gamma_{m,n}$ in $\A^{d-s+1+m+n}$. Our aim is to provide explicit
equations defining $\mathrm{cl}(\Gamma_{m,n})$. For this purpose,
let $\Gamma^*_{m,n}\subset \A^{d-s+1+m+n}$ be the affine
$\fq$--variety defined as
\begin{eqnarray*}
\Gamma_{m,n}^*:=\{(\boldsymbol{b},b_{0,1},b_{0,2},\boldsymbol{\alpha},\boldsymbol{\beta})
\in\A^{d-s+1+m+n}:\Delta^{i-1}F(\boldsymbol{b},b_{0,1},
\alpha_1\klk\alpha_i)=0\\ (1\leq i\leq m),\,
\Delta^{j-1}F({\boldsymbol{b},b_{0,2},\beta_1\klk\beta_j})=0\ (1\leq
j\leq n)\},
\end{eqnarray*}
where $\Delta^{i-1}F(\boldsymbol{b},b_{0,1},T_1,\ldots,T_i)$ and
$\Delta^{j-1}F(\boldsymbol{b},b_{0,2},U_1,\ldots,U_j)$ denote the
divided differences of  $F(\boldsymbol{b},b_{0,1},T)\in \cfq[T]$ and
$F(\boldsymbol{b},b_{0,2},U)\in \cfq[U]$ respectively. The relation
between the varieties $\Gamma_{m,n}$ and $\Gamma_{m,n}^*$ is
expressed in the following result.
\begin{lemma}\label{lemma: relacion gamma_m,n y gamma_m,n estrella}
With notations and assumptions as above, we have the following
identity:
\begin{equation}\label{eq: relacion gamma_m,n y gamma_m,n estrella}
\Gamma_{m,n}=\Gamma_{m,n}^*\cap\{\alpha_i\neq\alpha_j\ (1\!\le
i\!<\!j\le m),\ \beta_i\neq\beta_j\ (1\!\le i\!<\!j\le n),\
b_{0,1}\neq b_{0,2}\}.
\end{equation}
\end{lemma}
\begin{proof}
This is an easy consequence of Lemma \ref{lemma: relacion gamma_r y
gamma_r estrella}.
\end{proof}
%
%
\section{Geometry of the variety $\Gamma_{m,n}^*$}
\label{sec: geometry of Gamma_mn estrella}
Let be given $m$ and $n$ with $1\le m,n\le d$ and $d-s+1\le m+n\le
2d$. In this section we obtain critical information on the geometry
of $\Gamma_{m,n}^{*}$, which shall allow us to conclude that
$\Gamma_{m,n}^*$ is the Zariski closure of $\Gamma_{m,n}$.

Several arguments in this section are similar to those of Section
\ref{sec: geometry of Gamma_r estrella}. Therefore, in order to
avoid repetitions, some proofs shall only be sketched.
\begin{lemma}
\label{lemma: gamma_m,n estrella es set theoret complete inter}
The variety $\Gamma_{m,n}^{*}$ is a set--theoretic complete
intersection of dimension $d-s+1$.
\end{lemma}
\begin{proof}
Consider the graded lexicographic order of
$\cfq[\boldsymbol{B},B_{0,1},B_{0,2},\boldsymbol{T},\boldsymbol{U}]$
with
$U_n>\cdots>U_1>T_m>\cdots>T_1>B_{d-s-1}>\cdots>B_{0,1}>B_{0,2}$.
Arguing as in Lemma \ref{lemma: Gamma r is set-theoret complete
intersection} it is easy to see that the leading terms of
$\Delta^{i-1}F(\boldsymbol{B}_1,T_1,\ldots,T_i)$ and
$\Delta^{j-1}F(\boldsymbol{B}_2,U_1,\ldots,U_j)$ are $T_i^{d-i+1}$
and $U_j^{d-j+1}$ respectively. This shows that
$\Delta^{i-1}F(\boldsymbol{B}_1,T_1,\ldots,T_i)$ $(1\le i\le m)$ and
$\Delta^{j-1}F(\boldsymbol{B}_2,U_1,\ldots,U_j)$ $(1\le j\le n)$
form a Gr\"obner basis of the ideal $\mathcal{J}_{m,n}$ that they
generate (see, e.g., \cite[\S 2.9, Proposition 4]{CoLiOS92}).
Furthermore, since the leading terms of
$\Delta^{i-1}F(\boldsymbol{B}_1,T_1,\ldots,T_i)$ $(1\le i\le m)$ and
$\Delta^{j-1}F(\boldsymbol{B}_2,U_1,\ldots,U_j)$ $(1\le j\le n)$
form a regular sequence, by \cite[Proposition 15.15]{Eisenbud95} we
conclude that $\Delta^{i-1}F(\boldsymbol{B}_1,T_1,\ldots,T_i)$
$(1\le i\le m)$ and $\Delta^{j-1}F(\boldsymbol{B}_2,U_1,\ldots,U_j)$
$(1\le j\le n)$ also form a regular sequence. As a consequence,
$\Gamma_{m,n}^{*}$ is a set--theoretic complete intersection of
dimension $d-s+1$.
\end{proof}

%
%
\subsection{The singular locus of $\Gamma_{m,n}^*$}
The aim of this section is to prove that the singular locus of
$\Gamma_{m,n}^*$ has codimension at least 2 in $\Gamma_{m,n}^*$.

Arguing as in the proof of Lemma \ref{lemma: jacobian_F full rank
implies nonsingular} it is easy to see that the polynomials
$F(\boldsymbol{B}_1,T_i)$ $(1\le i\le m)$ and
$F(\boldsymbol{B}_2,U_j)$ $(1\leq j\leq n)$ vanish on
$\Gamma_{m,n}^*$. As a consequence, we have the following criterion
of nonsingularity.
\begin{remark}\label{rem: criterion nonsingularity variance}
Let $J_{F_{1,2}}$ be the Jacobian matrix of the polynomials
$F(\boldsymbol{B}_1,T_i)$ $(1 \leq i \leq m)$ and
$F(\boldsymbol{B}_2,U_j)$ $(1 \leq j \leq n)$ with respect to
$\boldsymbol{B},B_{0,1},B_{0,2},\boldsymbol{T},\boldsymbol{U}$. If
$(\boldsymbol{b},b_{0,1},b_{0,2},\boldsymbol{\alpha},
\boldsymbol{\beta})\in \Gamma_{m,n}^*$ is such that $\rank
J_{F_{1,2}}(\boldsymbol{b},b_{0,1},b_{0,2},\boldsymbol{\alpha},
\boldsymbol{\beta})=m+n$, then it is nonsingular.
\end{remark}
Let $(\boldsymbol{b},b_{0,1},b_{0,2},\boldsymbol{\alpha},
\boldsymbol{\beta})$ be an arbitrary point of $\Gamma_{m,n}^*$, with
$\boldsymbol{\alpha}:=(\alpha_1,\ldots,\alpha_m)$ and
$\boldsymbol{\beta}:=(\beta_1,\ldots,\beta_n)$. Denote
$\bfs{b}_1:=(\bfs{b},b_{0,1})$ and $\bfs{b}_2:=(\bfs{b},b_{0,2})$.
Then specializing the Jacobian $J_{F_{1,2}}$ at
$(\boldsymbol{b},b_{0,1},b_{0,2},\boldsymbol{\alpha},
\boldsymbol{\beta})$ we obtain the following matrix:
\begin{equation}\label{eq: Jacobian matrix fb1 and fb2}
J_{F_{1,2}}(\boldsymbol{b},b_{0,1},b_{0,2},\boldsymbol{\alpha},
\boldsymbol{\beta}):=\left(\!\!\begin{array}{ccccccccccccc}
\alpha_1^{d-s-1}  \!&\! \ldots \!&\! \alpha_1 \!&\!  1     \!&\! 0      \!&\! \gamma_1 \!&\! 0      \!&\! \ldots   \!&\! 0      \!&\! \cdots \!&\! 0  \\
  \vdots          \!&\!        \!&\! \vdots   \!&\! \vdots \!&\! \vdots \!&\! \ddots   \!&\! \ddots \!&\! \ddots   \!&\! \vdots \!&\!        \!&\! \vdots\\
 \alpha_m^{d-s-1} \!&\! \ldots \!&\! \alpha_m \!&\!  1     \!&\! 0      \!&\! \ldots   \!&\! 0      \!&\! \gamma_m \!&\!  0     \!&\! \ldots \!&\! 0 \\
  \beta_1^{d-s-1} \!&\! \ldots \!&\! \beta_1  \!&\!  0     \!&\! 1      \!&\! 0        \!&\! \ldots \!&\!  0       \!&\! \eta_1 \!&\! \ddots \!&\! \vdots\\
    \vdots        \!&\!        \!&\! \vdots   \!&\! \vdots \!&\! \vdots \!&\! \vdots   \!&\!        \!&\! \vdots   \!&\! \ddots \!&\! \ddots \!&\! 0 \\
  \beta_n^{d-s-1} \!&\! \ldots \!&\! \beta_n  \!&\!  0     \!&\! 1      \!&\! 0        \!&\! \ldots \!&\! 0        \!&\! \ldots \!&\!  0     \!&\! \eta_n
\end{array}\!\!\right),
\end{equation}
where $\gamma_i:=f_{\boldsymbol{b}}'(\alpha_i)$ and
$\eta_j:=f_{\boldsymbol{b}}'(\beta_j)$ for $1\le i\le m$ and $1\le
j\le n$. Therefore, from Remark \ref{rem: criterion nonsingularity
variance} we immediately deduce the following remark.
\begin{remark}
\label{rem: at most one multiple root implies nonsingular}
If there exist at most one $\alpha_i$ and at most one $\beta_j$
which are multiple roots of $f_{\boldsymbol{b}_1}$ and
$f_{\boldsymbol{b}_2}$ respectively, then
$(\boldsymbol{b},b_{0,1},b_{0,2},\boldsymbol{\alpha},
\boldsymbol{\beta})$ is a nonsingular point of $\Gamma_{m,n}^*$.
\end{remark}

Consider the following morphism of $\fq$--varieties:
\begin{equation}\label{eq: morfismo finito varianza}
\begin{array}{rccl}
\Psi_{m,n}:& {\Gamma}_{m,n}^{*}& \longrightarrow& \A^{d-s+1}\\
   &(\boldsymbol{b},b_{0,1},b_{0,2},\boldsymbol{\alpha},
   \boldsymbol{\beta})& \mapsto&(\boldsymbol{b},b_{0,1},b_{0,2}).
\end{array}
\end{equation}
%
%
Arguing as in the proof of Lemma \ref{lemma: f'=0} we easily deduce
that $\Psi_{m,n}$ is a finite morphism.

Let $(\boldsymbol{b},b_{0,1},b_{0,2},\boldsymbol{\alpha},
\boldsymbol{\beta})$ be a singular point of $\Gamma_{m,n}^*$.
According to Remark \ref{rem: at most one multiple root implies
nonsingular}, either $f_{\boldsymbol{b}_1}$ or
$f_{\boldsymbol{b}_2}$ has multiple roots. We observe that we may
assume without loss of generality that $f_{\bfs{b}}'\not=0$ and
$\Delta^2F(\bfs{b}_1,T,T,T)\not=0$. More precisely, from the proofs
of Lemmas \ref{lemma: f'=0} and \ref{lemma: repeated multiple root
II - mean} we deduce the following remark.
\begin{remark}\label{rem: f'=0 or f''=0}
If either $d-s\geq 4$ and $p>3$, or $d-s\ge 6$ and $p=3$, then the
set $\mathcal{W}_1'$ of points
$(\boldsymbol{b},b_{0,1},b_{0,2},\boldsymbol{\alpha},
\boldsymbol{\beta})$ of $\Gamma_{m,n}^*$ such that
$f'_{\boldsymbol{b}}=0$ or $\Delta^2F(\bfs{b}_1,T,T,T)=0$ is
contained in a subvariety of $\Gamma_{m,n}^*$ of codimension $2$.
\end{remark}

Next we study the set of singular points of $\Gamma_{m,n}^*$ for
which $f_{\bfs{b}}'\not=0$. We first consider the points for which
$f_{\bfs{b}_1}$ and $f_{\bfs{b}_2}$ have multiple roots in $\cfq$.
\begin{lemma}
\label{lemma: one multiple alfa and one multiple beta - variance}
Let $\mathcal{W}'\subset\Gamma_{m,n}^*$ be the set of points
$(\boldsymbol{b},b_{0,1},b_{0,2},\boldsymbol{\alpha},
\boldsymbol{\beta})$ such that $f_{\bfs{b}}'\not=0$ and
$f_{\bfs{b}_1}$ and $f_{\bfs{b}_2}$ have multiple roots in $\cfq$.
Then $\mathcal{W}'$ is contained in a codimension--2 subvariety of
$\Gamma_{m,n}^*$
\end{lemma}
\begin{proof}
Let $(\boldsymbol{b},b_{0,1},b_{0,2},\boldsymbol{\alpha},
\boldsymbol{\beta})$ be an arbitrary point of $\mathcal{W}'$. Then
we have that
$\mathrm{Res}(f_{\boldsymbol{b}_1},f'_{\boldsymbol{b}_1}) =
\mathrm{Res}(f_{\boldsymbol{b}_2},f'_{\boldsymbol{b}_2})=0$, where
$\mathrm{Res}(f_{\boldsymbol{b}_l} ,f'_{\boldsymbol{b}_l})$ denotes
the resultant of $f_{\boldsymbol{b}_l}$ and $f'_{\boldsymbol{b}_l}$.
Since $f_{\boldsymbol{b}_1}$ and $f_{\boldsymbol{b}_2}$ have degree
$d$ and $f'_{\boldsymbol{b}_1}$ and $f'_{\boldsymbol{b}_2}$ are
nonzero polynomials, it follows that
\begin{eqnarray*}
\mathrm{Res}(f_{\boldsymbol{b}_1},
f'_{\boldsymbol{b}_1})&=&\mathrm{Res}(F(\boldsymbol{B}_1,
T_1),\Delta^1F(\boldsymbol{B}_1,T_1,T_1),T_1)|_{\bfs{B}_1=\boldsymbol{b}_1},\\
\mathrm{Res}(f_{\boldsymbol{b}_2},
f'_{\boldsymbol{b}_2})&=&\mathrm{Res}(F(\boldsymbol{B}_2,U_1),
\Delta^1F(\boldsymbol{B}_2,U_1,U_1),U_1)|_{\bfs{B}_2=\boldsymbol{b}_2}.
\end{eqnarray*}
Let $\mathcal{R}_1:=\mathrm{Res}(F(\boldsymbol{B}_1,T_1),
\Delta^1F(\boldsymbol{B}_1,T_1,T_1),T_1)$ denote the resultant of
the polynomials $F(\boldsymbol{B}_1,T_1)$ and
$\Delta^1F(\boldsymbol{B}_1,T_1,T_1)$ with respect to $T_1$ and let
$\mathcal{R}_2:=\mathrm{Res}(F(\boldsymbol{B}_2,U_1),
\Delta^1F(\boldsymbol{B}_2,U_1,U_1),U_1)$ denote the resultant of
$F(\boldsymbol{B}_2,U_1)$ and $\Delta^1F(\boldsymbol{B}_2,U_1,U_1)$
with respect to $U_1$. Then
$\mathcal{W}'\subset\Psi_{m,n}^{-1}(\mathcal{Z})$, where
$\Psi_{m,n}$ is the morphism of (\ref{eq: morfismo finito varianza})
and $\mathcal{Z}\subset\A^{d-s+1}$ is the subvariety of $\A^{d-s+1}$
defined by the equations
$$\mathcal{R}_1(\boldsymbol{B}_1)=
\mathcal{R}_2(\boldsymbol{B}_2)=0.$$

Since $F(\boldsymbol{B}_1,T_1)$ is a separable element of
$\fq[\bfs{B}_1][T_1]$, the resultant $\mathcal{R}_1$ is nonzero
element of $\fq[\boldsymbol{B}_1]$. Furthermore, from, e.g.,
\cite[\S 1]{FrSm84}, one deduces that $\mathcal{R}_1$ is an element
of $\fq[\bfs{B}][B_{0,1}]\setminus\fq[\bfs{B}]$. Analogously,
$\mathcal{R}_2$ is a nonconstant polynomial of
$\fq[\bfs{B}][B_{0,2}]$. According to Theorem \ref{th: irred
discrim}, $\mathcal{R}_1$ is an irreducible element of
$\fq[\bfs{B}][B_{0,1},B_{0,2}]$ and
$\mathcal{R}_2\in\fq[\bfs{B}][B_{0,1},B_{0,2}]$ is not a multiple of
$\mathcal{R}_1(\boldsymbol{B}_1)$ in
$\fq[\bfs{B}][B_{0,1},B_{0,2}]$. This implies that
$\mathcal{R}_1(\boldsymbol{B}_1)$ and
$\mathcal{R}_2(\boldsymbol{B}_2)$ form a regular sequence in
$\cfq[\boldsymbol{B},B_{0,1},B_{0,2}]$. It follows that the variety
$\mathcal{Z}$ has dimension $d-s-1$, and hence
$\dim\Psi_{m,n}^{-1}(\mathcal{Z})=d-s-1$. This finishes the proof of
the lemma.
\end{proof}

According to Lemma \ref{lemma: one multiple alfa and one multiple
beta - variance} it remains to analyze the set of singular points
$(\boldsymbol{b},b_{0,1},b_{0,2},\boldsymbol{\alpha},
\boldsymbol{\beta})$ of $\Gamma_{m,n}^*$ for which either
$f_{\boldsymbol{b}_1}$, or $f_{\boldsymbol{b}_2}$, has only simple
roots in $\cfq$. In what follows we shall assume without loss of
generality that the latter case holds. By Remark \ref{rem: at most
one multiple root implies nonsingular} there must be at least two
distinct coordinates of $\bfs{\alpha}$ which are multiple roots of
$f_{\boldsymbol{b}_1}$.

Suppose first that there exist two coordinates of $\bfs{\alpha}$
whose values are two distinct multiple roots of $f_{\bfs{b}_1}$.
Arguing as in Lemma \ref{lemma: two distinct multiple roots - mean}
we easily deduce the following remark.
\begin{remark}\label{rem: two distinct multiple roots in fb1 - variance}
Let $\mathcal{W}_2'$ denote the set of points
$(\boldsymbol{b},b_{0,1},b_{0,2},\bfs{\alpha},\bfs{\beta})\in
\Gamma_{m,n}^{*}$ for which the following conditions hold:
\begin{itemize}
  \item $f_{\boldsymbol{b}_2}$ has only simple roots in $\cfq$,
  \item there exist $1\le i<j\le m$ such that $\alpha_i\not=\alpha_j$ and
$\alpha_i,\alpha_j$ are multiple roots of $f_{\boldsymbol{b}_1}$.
\end{itemize}
Then $\mathcal{W}_2'$ is contained in a subvariety of codimension 2
of $\Gamma_{m,n}^{*}$.
\end{remark}

Next we consider the points of $\Gamma_{m,n}^*$ for which there
exist exactly two distinct coordinates of $\bfs{\alpha}$ whose value
is a multiple root of $f_{\boldsymbol{b}_1}$, and both take the same
value. Arguing as in Lemma \ref{lemma: repeated multiple root I -
mean} we obtain the following remark.
\begin{remark}\label{rem: repeated multiple root I - variance}
Let $(\boldsymbol{b},b_{0,1},b_{0,2},\bfs{\alpha},\bfs{\beta})\in
\Gamma_{m,n}^{*}$ be a point satisfying the following conditions:
\begin{itemize}
\item $f_{\boldsymbol{b}_2}$ has only simple roots in $\cfq$;
\item there exist $1\le i<j\le m$ such that $\alpha_i=\alpha_j$ and
$\alpha_i$ is a multiple root of $f_{\boldsymbol{b}_1}$;
\item for any $k\notin\{i,j\}$, $\alpha_k$ is a simple root of
$f_{\boldsymbol{b}_1}$.
\end{itemize}
Then $(\boldsymbol{b},b_{0,1},b_{0,2},\bfs{\alpha},\bfs{\beta})$ is
regular point of $\Gamma_{m,n}^*$.
\end{remark}

Finally, we analyze the set of points of $\Gamma_{m,n}^*$ such that
there exist three distinct coordinates of $\bfs{\alpha}$ taking as
value the same multiple root of $f_{\boldsymbol{b}_1}$. By Lemma
\ref{lemma: repeated multiple root II - mean} we deduce the
following remark.
\begin{remark}\label{rem: repeated multiple root II - variance}
Let $\mathcal{W}_3'$ be the set of points
$(\boldsymbol{b},b_{0,1},b_{0,2},\boldsymbol{\alpha},\bfs{\beta})
\in\Gamma_{m,n}^*$ for which $f_{\bfs{b}_2}$ has only simple roots
in $\cfq$ and there exist $1\le i<j<k\le m$ such that
$\alpha_i=\alpha_j=\alpha_k$ and $\alpha_i$ is a multiple root of
$f_{\boldsymbol{b}_1}$. If either $d-s \geq 4$ and $p>3$, or $d-s\ge
6$ and $p=3$, then $\mathcal{W}_3'$ is contained in a codimension--2
subvariety of $\Gamma_{m,n}^*$.\end{remark}

Now we are able to obtain our lower bound on the codimension of the
singular locus of $\Gamma_{m,n}^*$. Combining Remarks \ref{rem: at
most one multiple root implies nonsingular}, \ref{rem: f'=0 or
f''=0}, \ref{rem: two distinct multiple roots in fb1 - variance},
\ref{rem: repeated multiple root I - variance} and \ref{rem:
repeated multiple root II - variance} and Lemma \ref{lemma: one
multiple alfa and one multiple beta - variance}, it follows that the
set of singular points of $\Gamma_{m,n}^*$ is contained in the set
$\mathcal{W}_1'\cup \mathcal{W}'\cup
\mathcal{W}_2'\cup\mathcal{W}_3'$, where $\mathcal{W}_1'$,
$\mathcal{W}'$, $\mathcal{W}_2'$ and $\mathcal{W}_3'$ are defined in
the statements of Remark \ref{rem: f'=0 or f''=0}, Lemma \ref{lemma:
one multiple alfa and one multiple beta - variance} and Remarks
\ref{rem: two distinct multiple roots in fb1 - variance} and
\ref{rem: repeated multiple root II - variance} respectively. Since
such a union of sets is contained in codimension--2 subvariety of
$\Gamma_{m,n}^*$, we obtain the following result.
\begin{theorem} \label{th: codimension singular locus Gamma estrella}
If either $d-s \geq 4$ and $p>3$, or $d-s\ge 6$ and $p=3$, then the
singular locus of $\Gamma_{m,n}^*$ has codimension at least $2$ in
$\Gamma_{m,n}^*$.
\end{theorem}

We finish this section with a consequence of the analysis underlying
the proof of Theorem \ref{th: codimension singular locus Gamma
estrella}. As the proof of this result is similar to that of
Corollary \ref{coro: J is radical}, it shall only be sketched.
\begin{corollary}
With assumptions be as in Theorem \ref{th: codimension singular
locus Gamma estrella}, let $\mathcal{J}_{m,n}\subset
\fq[\bfs{B},B_{0,1},B_{0,2},\bfs{T},\bfs{U}]$ be the ideal generated
by $\Delta^{i-1}F({\boldsymbol{B}_1,T_1,\ldots,T_i})$ $(1\leq i\leq
m)$ and $\Delta^{j-1}F({\boldsymbol{B}_2,U_1,\ldots,U_j})$ $(1\leq
j\leq n)$. Then $\mathcal{J}_{m,n}$ is a radical ideal.
\end{corollary}
\begin{proof}
By Lemma \ref{lemma: gamma_m,n estrella es set theoret complete
inter}, the polynomials
$\Delta^{i-1}F({\boldsymbol{B}_1,T_1,\ldots,T_i})$ $(1\leq i\leq m)$
and $\Delta^{j-1}F({\boldsymbol{B}_2,U_1,\ldots,U_j})$ $(1\leq j\leq
n)$ form a regular sequence. Let $J_{\Delta_{1,2}}$ be Jacobian
matrix of these polynomials with respect to
$\bfs{B},B_{0,1},B_{0,2},\bfs{T},\bfs{U}$. We claim that the set of
points
$(\boldsymbol{b},b_{0,1},b_{0,2},\boldsymbol{\alpha},\bfs{\beta})\in
\Gamma_{m,n}^*$ for which the Jacobian matrix
$J_{\Delta_{1,2}}(\boldsymbol{b},b_{0,1},
b_{0,2},\boldsymbol{\alpha},\bfs{\beta})$ has not full rank has
codimension at least 1 in $\Gamma_{m,n}^*$. Indeed, if
$J_{\Delta_{1,2}}(\boldsymbol{b},b_{0,1},
b_{0,2},\boldsymbol{\alpha},\bfs{\beta})$ has not full rank, then
the matrix
$J_{F_{1,2}}(\boldsymbol{b},b_{0,1},b_{0,2},\boldsymbol{\alpha},\bfs{\beta})$
of (\ref{eq: Jacobian matrix fb1 and fb2}) has not full rank. On the
other hand, the latter implies that $f_{\boldsymbol{b}_1}$ or
$f_{\boldsymbol{b}_2}$ has multiple roots in $\cfq$. Therefore, by
the arguments of the proofs of Remark \ref{rem: f'=0 or f''=0} and
Lemma \ref{lemma: one multiple alfa and one multiple beta -
variance} we deduce the claim. As a consequence, the statement of
the corollary is readily implied by \cite[Theorem
18.15]{Eisenbud95}.
\end{proof}
%
%
\subsection{The geometry of the projective closure of $\Gamma_{m,n}^*$}
Similarly to Section \ref{subsec: geometry of pcl Gamma r}, in this
section we discuss the behavior of $\Gamma_{m,n}^*$ at infinity. For
this purpose, we shall consider the projective closure of
$\mathrm{pcl}(\Gamma_{m,n}^*)\subset\Pp^{d-s+1+m+n}$ of
$\Gamma_{m,n}^*$, and the set of $\mathrm{pcl}(\Gamma_{m,n}^*)$ at
infinity, namely the points of $\mathrm{pcl}(\Gamma_{m,n}^*)$ lying
in the hyperplane $\{T_0=0\}$.

Let $\mathcal{J}_{m,n}^h\subset
\fq[\bfs{B},B_{0,1},B_{0,2},T_0,\bfs{T},\bfs{U}]$ be the ideal
generated by the homogenizations $F^h$ of all the polynomials $F\in
\mathcal{J}_{m,n}$.
\begin{lemma}
\label{lemma: pcl Gamma mn is ideal-theoret complete inters} With
assumptions on $d$, $s$ and $p$ as in Theorem \ref{th: codimension
singular locus Gamma estrella}, the homogenized polynomials
$\Delta^{i-1}F(\boldsymbol{B}_1,T_1,\ldots,T_i)^h$ $(1\leq i \leq
m)$  and\linebreak
$\Delta^{j-1}F(\boldsymbol{B}_2,U_1,\ldots,U_j)^h$ $(1\leq j \leq
n)$ generate the ideal $\mathcal{J}_{m,n}^h$. Furthermore,
$\mathrm{pcl}(\Gamma_{m,n}^*)$ is an ideal-theoretic complete
intersection of dimension $d-s+1$ and degree $(d!)^2/(d-m)!(d-n)!$.
\end{lemma}
\begin{proof}
The proof of the lemma is deduced {\em mutatis mutandis} following
the proof of Lemma \ref{lemma: pcl Gamma r is ideal-theoret complete
inters}, considering the graded lexicographical order of \linebreak
$\fq[\bfs{B},B_{0,1},B_{0,2},\bfs{T},\bfs{U}]$ defined by
$U_n>\cdots> U_1>T_m>\cdots >T_1>B_{d-s-1}>\cdots
>B_1>B_{0,1}>B_{0,2}$.
\end{proof}

Similarly to Lemma \ref{lemma: singular locus pcl Gamma_r at
infinity}, the set of points of $\mathrm{pcl}(\Gamma_{m,n}^*)$ at
infinity is a linear variety. We shall skip the the proof of this
result, because it is similar to that of Lemma \ref{lemma: singular
locus pcl Gamma_r at infinity}.
\begin{lemma}\label{lemma: singular locus pcl Gamma_mn at infinity}
$\mathrm{pcl}(\Gamma_{m,n}^*)\cap \{T_0=0\}\subset \Pp^{d-s+m+n}$ is
a linear $\fq$--variety of dimension $d-s$.
\end{lemma}

Combining Theorem \ref{th: codimension singular locus Gamma
estrella} and Lemmas \ref{lemma: pcl Gamma mn is ideal-theoret
complete inters} and \ref{lemma: singular locus pcl Gamma_mn at
infinity} as we did in the proof of Theorem \ref{th: pcl Gamma_r is
normal abs irred} we obtain the main result of this section.
\begin{theorem}
\label{th: pcl Gamma_mn is normal abs irred} With assumptions on
$d$, $s$ and $p$ as in Theorem \ref{th: codimension singular locus
Gamma estrella}, the projective variety
$\mathrm{pcl}(\Gamma_{m,n}^*)\subset \Pp^{d-s+1+m+n}$ is a normal
absolutely irreducible ideal-theoretic complete intersection defined
over $\fq$ of dimension $d-s+1$ and degree $(d!)^2/(d-m)!(d-n)!$.
\end{theorem}

We deduce that  $\Gamma_{m,n}^* \subset \A^{d-s+1+m+n}$ is an
absolutely irreducible ideal-theoretic complete intersection of
dimension $d-s+1$ and degree $(d!)^2/(d-m)!(d-n)!$. Furthermore,
Lemma \ref{lemma: relacion gamma_m,n y gamma_m,n estrella} shows
that $\Gamma_{m,n}$ coincides with the subset of points of
$\Gamma_{m,n}^*$ with $b_{0,1} \neq b_{0,2}$, $\alpha_i \neq
\alpha_j$ and $\beta_k\neq \beta_{l}$. Hence, taking into account
that $\Gamma_{m,n}^*$ is absolutely irreducible, we deduce that
$\mathrm{cl}(\Gamma_{m,n})=\Gamma_{m,n}^{*}$.
%
%
\section{The asymptotic behavior of $\mathcal{V}_2(d,s,\bfs{a})$}
\label{sec: behavior V2}
As before, let be given positive integers $d$ and $s$ such that,
either $d-s \geq 4$ and $p>3$, or $d-s\ge 6$ and $p=3$. As asserted
before, our objective is to determine the asymptotic behavior of the
quantity $\mathcal{V}_2(d,s,\boldsymbol{a})$ of (\ref{eq: formula
para v2}) for a given $\boldsymbol{a}:=(a_{d-1},\ldots,a_{d-s})\in
\fq^{s}$. According to Theorem \ref{th: combinatorial reduction
variance}, such an asymptotic behavior is determined by that of the
number $\mathcal{S}_{m,n}^{\boldsymbol{a}}$ defined in (\ref{eq: S
mn como suma sobre los gammas}) for each pair $(m,n)$ with $1\leq
m,n \leq d$ and $d-s+1\leq m+n \leq 2d$.
%
%
\subsection{An estimate for $\mathcal{S}_{m,n}^{\boldsymbol{a}}$}
\label{subsec: estimate S mn}
Lemma \ref{lemma: relacion entre gamma m,n y S m,n} expresses
$\mathcal{S}_{m,n}^{\boldsymbol{a}}$ in terms of the number of
$q$--rational points of the affine quasi--$\fq$--variety
$\Gamma_{m,n}$. As a consequence, we estimate the number
$|\Gamma_{m,n}(\fq)|$ of $q$--rational points of $\Gamma_{m,n}$ for
each pair $(m,n)$ as above.

Lemma \ref{lemma: relacion gamma_m,n y gamma_m,n estrella} relates
the quantity $|\Gamma_{m,n}(\fq)|$ with the number
$|\Gamma_{m,n}^*(\fq)|$ of $q$--rational points of the affine
$\fq$--variety $\Gamma_{m,n}^*$. We shall express the latter in
terms of the number of $q$--rational points of the projective
closure $\mathrm{pcl}(\Gamma_{m,n}^*)$ and its set
$\mathrm{pcl}(\Gamma_{m,n}^*)^{\infty}:=\mathrm{pcl}(\Gamma_{m,n}^*)\cap
\{T_0=0\}$ of points at infinity.

Theorem \ref{th: pcl Gamma_mn is normal abs irred} shows that
$\mathrm{pcl}(\Gamma_{m,n}^*)$ is a normal complete intersection of
dimension $d-s+1$ defined over $\fq$, and therefore (\ref{eq:
estimate normal var CaMaPr}) yields the following estimate:
$$
\big||\mathrm{pcl}(\Gamma_{m,n}^*)(\fq)|-p_{d-s+1}\big| \leq
\big(\delta_{m,n}(D_{m,n}-2)+2\big)q^{d-s+\frac{1}{2}}+
14D_{m,n}^2\delta_{m,n}^2q^{d-s},
$$
where $D_{m,n}:=\sum_{i=1}^m(d-i)+\sum_{j=1}^n
(d-j)=(m+n)d-\big(m(m+1)+n(n+1)\big)/{2}$ and
$\delta_{m,n}:=(d!)^2/(d-m)!(d-n)!$. On the other hand, Lemma
\ref{lemma: singular locus pcl Gamma_mn at infinity} proves that
$\mathrm{pcl}(\Gamma_{m,n}^*)^{\infty}$ is linear $\fq$--variety of
dimension $d-s$. Thus we obtain
\begin{eqnarray}\label{eq: estimate q points pcl Gamma mn}
\big||\Gamma_{m,n}^*(\fq)|-q^{d-s+1}\big|&=&
\big||\mathrm{pcl}(\Gamma_{m,n}^*)|-|\mathrm{pcl}(\Gamma_{m,n}^*)^{\infty}|-p_{d-s+1}+p_{d-s}\big|\nonumber\\
&\leq& (\delta_{m,n}(D_{m,n}-2)+2)q^{d-s+\frac{1}{2}}
+14D_{m,n}^2\delta_{m,n}^2q^{d-s}.
\end{eqnarray}

Next we estimate the number $|\Gamma_{m,n}(\fq)|$. For this purpose,
according to Lemma \ref{lemma: relacion gamma_m,n y gamma_m,n
estrella} we obtain an upper bound on the number of $q$--rational
points $(\bfs{b},b_{0,1},b_{0,2},\bfs{\alpha},\bfs{\beta})$ of
$\Gamma_{m,n}^*$ such that, either $b_{0,1}=b_{0,2}$, or there exist
$1\le i<j\le m$ with $\alpha_i=\alpha_j$, or there exist $1\le
k<l\le n$ with $\beta_k=\beta_l$. Such a subset of $\Gamma_{m,n}^*$
form the following $\fq$-variety:
\begin{equation*}
\Gamma_{m,n}^{*,\,=}:= \Gamma_{m,n}^* \cap
\bigg(\left\{B_{0,1}=B_{0,2}\right\}\bigcup_{1 \leq i <j \leq m}
\left\{T_i=T_j\right\} \bigcup_{1 \leq k <l \leq
n}\left\{U_k=U_l\right\}\bigg).
\end{equation*}
Observe that $\Gamma_{m,n}^{*,\,=}=\Gamma_{m,n}^{*}\cap
\mathcal{H}_{m,n}$, where $\mathcal{H}_{m,n}\subset \A^{d-s+1+m+n}$
is the hypersurface defined by the polynomial
$$F:=(B_{0,1}-B_{0,2})\prod_{1
\leq i <j \leq m} (T_i-T_j) \prod_{1 \leq k <l \leq n}(U_k-U_l).$$
By the B\'ezout inequality (\ref{eq: Bezout inequality}) we have
\begin{equation}\label{eq: grado de gamma igual}
\deg \Gamma_{m,n}^{*,\,=} \leq \delta_{m,n}\left( \binom {m}{2}+
\binom {n}{2}+1\right),
\end{equation}
The set $\Gamma_{m,n}^*\cap\{B_{0,1}=B_{0,2}\}$ is contained in the
codimension--1 subvariety of $\Gamma_{m,n}^*$ given by
$\Psi_{m,n}^{-1}(\{B_{0,1}=B_{0,2}\})$. Furthermore, if
$\alpha_i=\alpha_j$ for $1\le i<j\le m$, then $\alpha_i$ is a
multiple root of $f_{\bfs{b}_1}$, and the same can be said of
$f_{\bfs{b}_2}$ if $\beta_k=\beta_l$ for $1\le k<l\le m$. Then, by
Remark \ref{rem: f'=0 or f''=0} and Lemma \ref{lemma: one multiple
alfa and one multiple beta - variance} we conclude that
$\Gamma_{m,n}^{*,\,=}$ has dimension at most $d-s$. Therefore,
combining, e.g., \cite[Lemma 2.1]{CaMa06} with (\ref{eq: grado de
gamma igual}) we obtain
\begin{equation}\label{eq: estimate q-points Gamma mn=}
\big|\Gamma_{m,n}^{*,\,=}(\fq)\big| \leq \delta_{m,n}
\left(\binom{m}{2} +\binom{n}{2}+1\right) q^{d-s}.
\end{equation}
Since $\Gamma_{m,n}(\fq)=(\Gamma_{m,n}^*)(\fq)\setminus
(\Gamma_{m,n}^{*,\,=})(\fq)$, from (\ref{eq: estimate q points pcl
Gamma mn}) and (\ref{eq: estimate q-points Gamma mn=}) we see that
\begin{eqnarray}\label{eq: estimate Gamma m,n}
\big||\Gamma_{m,n}(\fq)|- q^{d-s+1}\big|\leq
\big||\Gamma_{m,n}^*(\fq)|- q^{d-s+1}\big|+\big|(\Gamma_{m,n}^{*,\,=})(\fq)\big|\nonumber
\hskip2.5cm\\
\le(\delta_{m,n}(D_{m,n}-2)+2)q^{d-s+\frac{1}{2}}
+(14D_{m,n}^2\delta_{m,n}^2+\xi_{m,n}\delta_{m,n})q^{d-s},
\end{eqnarray}
where $\xi_{m,n} :=\binom{m}{2}+\binom {n}{2}+1$.

Finally, by Lemma \ref{lemma: relacion entre gamma m,n y S m,n} and
(\ref{eq: estimate Gamma m,n}) we obtain the following result.
\begin{theorem}\label{th: estimate S mn}
Let be given positive integers $d$ and $s$ such that, either $d-s
\geq 4$ and $p>3$, or $d-s\ge 6$ and $p=3$. For each $(m,n)$ with
$1\leq m,n \leq d$, and $d-s+1\leq m+n\leq 2d$, we have
\begin{eqnarray*}
\bigg|\mathcal{S}_{m,n}^{\bfs{a}}-
\frac{q^{d-s+1}}{m!n!}\bigg|&\leq&
\frac{1}{m!n!}\big(\delta_{m,n}(D_{m,n}-2)+2\big)q^{d-s+\frac{1}{2}}\\
&&+
\frac{1}{m!n!}(14D_{m,n}^2\delta_{m,n}^2+\xi_{m,n}\delta_{m,n})q^{d-s},
\end{eqnarray*}
where $\xi_{m,n} :=\binom{m}{2}+\binom {n}{2}+1$,
$D_{m,n}:=(m+n)d-\binom{m+1}{2}-\binom{n+1}{2}$ and
$\delta_{m,n}:=\frac{(d!)^2}{(d-m)!(d-n)!}$.
\end{theorem}
%
%
\subsection{The asymptotic behavior of $\mathcal{V}_2(d,s,\boldsymbol{a})$}
\label{subsec: asymptotic behavior 2nd moment}
Theorem \ref{th: estimate S mn} is the fundamental step towards the
determination of the asymptotic behavior of
$\mathcal{V}_2(d,s,\boldsymbol{a})$. Indeed, by Theorem \ref{th:
combinatorial reduction variance} we have
\begin{eqnarray}
  \mathcal{V}_2(d,s,\boldsymbol{a})-\mu_d^2\,q^2=
  \mathcal{V}(d,s,\boldsymbol{a})+
  \hskip-0.2cm\mathop{\sum_{1\leq m,n\leq d}}_{2\leq m+n\leq d-s}\hskip-0.2cm (-q)^{2-m-n}
  \bigg(\binom{q}{n}\binom{q}{m}-\frac{q^{m+n}}{m!n!}\bigg)\nonumber  \\
   +\frac{1}{q^{d-s-1}}\hskip-0.4cm\mathop{\sum_{1\leq m,n\leq d}}_{d-s+1\leq m+n\leq 2d}
   \hskip-0.4cm(-1)^{m+n} \bigg(\mathcal{S}_{m,n}^{\boldsymbol{a}}-\frac{q^{d-s+1}}{m!n!}\bigg).
   \label{eq: igualdad de V_2(d,s,a)}
\end{eqnarray}

From Corollary \ref{coro: average value sets} it follows that
\begin{equation}\label{eq: upper bound V(d,s,a)} \mathcal{V}(d,s,\boldsymbol{a})\le
\mu_d\,q+d^2 2^{d-1}q^{1/2}+
\frac{7}{2}\,d^4\sum_{k=0}^{s-1}\binom{d}{k}^{2}(d-k)!.
\end{equation}

Next we obtain an upper bound for the absolute value $A_1(d,s)$ of
the second term in the right-hand side of (\ref{eq: igualdad de
V_2(d,s,a)}). Indeed, taking into account that
$$\binom{q}{n}\binom{q}{m}-\frac{q^{m+n}}{m!n!}=
\binom{q}{m}\bigg(\binom{q}{n}-\frac{q^{n}}{n!}\bigg)+\frac{q^n}{n!}
\bigg(\binom{q}{m}-\frac{q^{m}}{m!}\bigg),$$
we see that
\begin{eqnarray*}
  A_1(d,s)
    & \leq &
  \bigg|\mathop{\sum_{1\leq m,n\leq d}}_{2\leq m+n\leq d-s}
  (-q)^{2-m-n}\binom{q}{m}\bigg(\binom{q}{n}-\frac{q^{n}}{n!}\bigg)\bigg|\nonumber
   \\
   & &+\, \bigg|\displaystyle\mathop{\sum_{1\leq m,n\leq d}}_{2\leq m+n\leq d-s}
   \frac{(-1)^n}{n!}(-q)^{2-m}\bigg(\binom{q}{m}-\frac{q^{m}}{m!}\bigg)\bigg|.
\end{eqnarray*}
Arguing as in the proof of \cite[Corollary 14]{CeMaPePr13}, we have
that
$$\bigg|\sum_{n=1}^{d-s-m}(-q)^{1-n}\bigg(\binom{q}{n}-\frac{q^n}{n!}\bigg)\bigg|
\le \frac{1}{2\,e}+\frac{1}{2}+\frac{7}{q}\le d.$$
Therefore,
$$
  \bigg|\!\!\!\!\mathop{\sum_{1\leq m,n\leq d}}_{2\leq m+n\leq
  d-s}\!\!\!\!\!\!
  (-q)^{1-m-n}\binom{q}{m}\bigg(\binom{q}{n}-\frac{q^{n}}{n!}\bigg)\bigg|
  \le \!d\!\!\sum_{m=1}^{d-s-1}\binom{q}{m}q^{-m}\!\!\le\! d\bigg(\!1+\frac{1}{q}\!\bigg)^q \le  e\,d.
$$
On the other hand,
$$
\bigg|\mathop{\sum_{1\leq m,n\leq d}}_{2\leq m+n\leq
d-s}\frac{(-1)^n}{n!}(-q)^{1-m}\bigg(\binom{q}{m}-\frac{q^{m}}{m!}\bigg)\bigg|\le
d\sum_{n=1}^{d-s-1}\frac{1}{n!} \le e\,d.
$$
Combining the two previous bounds we obtain $A_1(d,s) \leq 2 \,e \,d
\,q$.

Finally, we consider the absolute value $B_1(d,s)$ of the last term
of (\ref{eq: igualdad de V_2(d,s,a)}). We have
\begin{eqnarray}
B_1(d,s)&\le& \sum_{m,n=1}^d
\frac{\delta_{m,n}(D_{m,n}-2)+2}{m!\,n!}\,q^{3/2}\nonumber\\&&
+14\sum_{m,n=1}^d
\frac{D_{m,n}^{2}\delta_{m,n}^{2}}{m!\,n!}\,q+\sum_{m,n=1}^d
\frac{\xi_{m,n}\delta_{m,n}}{m!\,n!}\,q. \label{eq: cota para
B_1(d,s)}
\end{eqnarray}
First we obtain an upper bound for the first term in the right--hand
side of the above inequality:
\begin{eqnarray}
\displaystyle\sum_{m,n=1}^d
\frac{\delta_{m,n}(D_{m,n}-2)+2}{m!\,n!}&\leq& 2\sum_{n=1}^{d}
\binom{d}{n}\frac{n(2d-n-1)}{2} \sum_{m=1}^{d}
\binom{d}{m}\nonumber\\[1ex]
&\leq& d^2 2^d(2^d-1). \label{eq: primer sumando de B_1(d,s)}
\end{eqnarray}
On the other hand, since $D_{m,n}^{2}\le (2d-1)^4/16$ for $1\le
m,n\le d$, we see that
\begin{eqnarray}
\displaystyle\sum_{m,n=1}^d
\frac{D_{m,n}^{2}\delta_{m,n}^{2}}{m!\,n!}&\le& \frac{1}{16}
(2d-1)^4\Bigg(\sum_{n=1}^{d}\binom{d}{n}^2n!\Bigg)^2\nonumber \\
&\le& \frac{1}{16} (2d-1)^4 \Bigg(\sum_{k=0}^{d-1} \binom{d}{k}^2
(d-k)!\Bigg)^2.\label{eq: segundo sumando de B_1(d,s)}
\end{eqnarray}
Finally, we consider the last term of (\ref{eq: cota para
B_1(d,s)}):
\begin{eqnarray}\label{eq: tercer sumando de B_1(d,s)}
\sum_{m,n=1}^d \frac{\delta_{m,n}\,\xi_{m,n}}{m!\,n!} &\le&
2\sum_{n=1}^d \binom{d}{n} \sum_{m=1}^d \binom{d}{m}\binom{m}{2}
 + \sum_{n=1}^d \binom{d}{n} \sum_{m=1}^d \binom{d}{m}\nonumber
 \\[1ex]&\leq& d^2 2^{d-2}(2^d-1).
\end{eqnarray}
Putting together (\ref{eq: primer sumando de B_1(d,s)}), (\ref{eq:
segundo sumando de B_1(d,s)}) and (\ref{eq: tercer sumando de
B_1(d,s)}) we obtain
$$B_1(d,s)\le d^2 2^{d-2}(2^d-1)(4 q^{3/2}+q) +\frac{7}{8}(2d-1)^4
\Bigg(\sum_{k=0}^{d-1}\binom{d}{k}^{2}(d-k)!\Bigg)^2q.$$

Combining (\ref{eq: upper bound V(d,s,a)}) and the upper bounds for
$A_1(d,s)$ and $B_1(d,s)$ above we deduce the following result.
\begin{corollary}
 \label{coro: estimate V2(d,s,a) without asympt behavior}
With assumptions and notations as in Theorem \ref{th: estimate S
mn}, we have
\begin{equation}
\label{eq: estimate V_2(d,s,a)}
\left|\mathcal{V}_2(d,s,\bfs{a})-\mu_d^2\, q^2\right|\le d^22^{2d+1}
q^{3/2}+14\,d^4\Bigg(\sum_{k=0}^{d-1}\binom{d}{k}^{2}(d-k)!\Bigg)^2q.
\end{equation}
\end{corollary}

We finish this section with a brief analysis of the behavior of the
right--hand side of (\ref{eq: estimate V_2(d,s,a)}). Such an
analysis is similar to that of Section \ref{subsec: behavior
V(d,s,a)}, and shall only be briefly sketched.

Fix $k$ with $0\le k\le d-1$ and denote
$h(k):=\binom{d}{k}{}^2(d-k)!$. Similarly to Remark \ref{rem: growth
h(k)}, it turns out that $h$ is a unimodal function in the integer
interval $[0,d-1]$ which reaches its maximum at $\lfloor
k_0\rfloor$, where $k_0:=-1/2+\sqrt{5+4d}/2$. As a consequence, we
see that
$$
\sum_{k=0}^{d-1}\binom{d}{k}^{2}(d-k)!\le d \binom{d}{\lfloor
k_0\rfloor}^2(d-\lfloor k_0\rfloor)!= \frac{d\,(d!)^2}{(d-\lfloor
k_0\rfloor)!\,(\lfloor k_0\rfloor!)^2}.
$$
With a similar analysis as in Section \ref{subsec: behavior
V(d,s,a)}, we conclude that
$$\Bigg(\sum_{k=0}^{d-1}\binom{d}{k}^{2}(d-k)!\Bigg)^2
\leq 14^2 d^{2d+2}e^{4\sqrt{d}-2d}.$$
Hence, we obtain the following result.
\begin{theorem}\label{th: final estimate V_2(d,s,a)}
With assumptions and notations as in Theorem \ref{th: estimate S
mn}, we have
$$
\left|\mathcal{V}_2(d,s,\bfs{a})-\mu_d^2\, q^2\right|\le
d^22^{2d+1}q^{3/2}+14^3 d^{2d+6}e^{4 \sqrt{d}-2d}q.$$
\end{theorem}
%
%
\section{On the second moment for $s=0$}\label{sec: V_2(d,0)}
As before, let be given a positive integer $d$ with $d<q$. In this
section we discuss how a similar analysis as the one underlying
Sections \ref{sec: combinatorial preliminaries variance}, \ref{sec:
geometric approach variance}, \ref{sec: geometry of Gamma_mn
estrella} and \ref{sec: behavior V2} allows us to establish the
asymptotic behavior of the quantity
$$\mathcal{V}_2(d,0):=\frac{1}{q^{d-1}}\sum_{\bfs{b}\in\fq^{d-1}}
\mathcal{V}(f_{\bfs{b}})^2,$$
namely the average second moment of $\mathcal{V}(f_{\bfs{b}})$ when
$f_{\bfs{b}}:=T^d+b_{d-1}T^{d-1}\plp b_1 T$ ranges over all monic
polynomials in $\fq[T]$ of degree $d$ with $f_{\bfs{b}}(0)=0$. As
stated in the introduction, an explicit expression for
$\mathcal{V}_2(d,0)$ is obtained for $d\ge q$ in \cite{KnKn90b}. On
the other hand, in \cite{Uchiyama56} it is shown that, for
$p:=\mathrm{char}(\fq)>d$ and assuming the Riemann hypothesis for
$L$--functions, one has
$\mathcal{V}_2(d,0)=\mu_d^2q^2+\mathcal{O}(q)$. It must be observed
that no explicit expression for the constant underlying the
$\mathcal{O}$--notation is provided in \cite{Uchiyama56}.

A similar argument as in the proof of Theorem \ref{th: combinatorial
reduction variance} yields the following result.
\begin{theorem}\label{th: combinatorial reduction variance s=0}
With assumptions and notations as above, we have
\begin{eqnarray*}\mathcal{V}_2(d,0)=
  \mathcal{V}(d,0)&+&
  \displaystyle\mathop{\sum_{1\leq m,n\leq d}}_{2\leq m+n\leq
  d}
  \binom{q}{m}\binom{q}{n}(-q)^{2-n-m}\\
 &+&\dfrac{1}{q^{d-1}}\mathop{\sum_{1\leq m,n\leq
d}}_{d+1\leq m+n\leq 2d}
   (-1)^{m+n}\mathop{\sum_{\Gamma_1,\Gamma_2 \subset\fq}}_{|\Gamma_1|=m,|\Gamma_2|=n}
 \left|S_{\Gamma_1,\Gamma_2}\right|, \end{eqnarray*}
where $\mathcal{S}_{\Gamma_1,\Gamma_2}$ is the set consisting of the
points $(\boldsymbol{b},b_{0,1},b_{0,2})\in \mathbb{F}_q^{d+1}$ with
$b_{0,1}\neq b_{0,2}$ such that
$(f_{\boldsymbol{b}}+b_{0,1})\big|_{\Gamma_1}\equiv 0$ and
$(f_{\boldsymbol{b}}+b_{0,2}){\big|_{\Gamma_2}}\equiv 0$ holds.
\end{theorem}

In view of Theorem \ref{th: combinatorial reduction variance s=0},
we fix $m$ and $n$ with $1\le m,n\le d$ and $d+1\le m+n\le 2d$ and
consider the sum
$$
\mathcal{S}_{m,n}:=\mathop{\sum_{\Gamma_1,\,\Gamma_2\subset
\fq}}_{|\Gamma_1|=m , |\Gamma_2|=n}|\mathcal{S}_{\Gamma_1,\,
\Gamma_2}|.
$$
In order to find an estimate for  $\mathcal{S}_{m,n}$ we introduce
new indeterminates $T,T_1,\ldots,T_m$, $U,U_1,\ldots,U_n$,
$B,B_{d-1},\ldots,B_1$, $B_{0,1}$, $B_{0,2}$ over $\cfq$ and denote
$\boldsymbol{B}:=(B_{d-1},\ldots,B_1)$. Furthermore, we consider the
polynomial $F:=T^d+\sum_{i=1}^{d-1}B_i T^i +B\in
\fq[\boldsymbol{B},B,T]$ and the following affine $\fq$--variety:
\begin{eqnarray*}
\Gamma_{m,n}^0:=\{(\boldsymbol{b},b_{0,1},b_{0,2},\boldsymbol{\alpha},\boldsymbol{\beta})
\in\A^{d+1+m+n}: \Delta^{i-1}F(\boldsymbol{b},b_{0,1},
\alpha_1\klk\alpha_i)=0\quad\\ (1\leq i\leq m),\
\Delta^{j-1}F({\boldsymbol{b},b_{0,2},\beta_1\klk\beta_j})=0\ (1\leq
j\leq n)\},
\end{eqnarray*}
where $\Delta^{i-1}F(\boldsymbol{b},b_{0,1},T_1,\ldots,T_i)$ and
$\Delta^{j-1}F(\boldsymbol{b},b_{0,2},U_1,\ldots,U_j)$ denote the
divided differences of $F(\boldsymbol{b},b_{0,1},T)\in \cfq[T]$ and
$F(\boldsymbol{b},b_{0,2},U)\in \cfq[U]$ respectively.

Arguing as in the proof of Lemmas \ref{lemma: relacion entre gamma
m,n y S m,n} and \ref{lemma: relacion gamma_m,n y gamma_m,n
estrella}, we conclude that the following identity holds:
\begin{eqnarray}
m!n!\mathcal{S}_{m,n}=\big|\Gamma_{m,n}^0(\fq)\cap\{\alpha_i\neq\alpha_j\
(1\le i<j\le m),\hskip2.1cm\nonumber\\
\label{eq: relation gamma mn y S mn s=0} \beta_i\neq\beta_j\ (1\le
i<j\le n),\ b_{0,1}\neq b_{0,2}\}\big|.\end{eqnarray}

The next step is to perform an analysis of the geometry of the
affine $\fq$--variety $\Gamma_{m,n}^0$, its projective closure
$\mathrm{pcl}(\Gamma_{m,n}^0)\subset\Pp^{d+1+m+m}$ and the set
$\mathrm{pcl}(\Gamma_{m,n}^0)^\infty$ of points of
$\mathrm{pcl}(\Gamma_{m,n}^0)$ at infinity. We refrain from giving
details, as proofs are similar to those of Theorems \ref{th:
codimension singular locus Gamma estrella} and \ref{th: pcl Gamma_mn
is normal abs irred}. We obtain the following result.
\begin{theorem}\label{th: geometry of Gamma mn s=0}
Assume that $d\ge 5$ for $p>3$ and $d\ge 9$ for $p=3$. Then the
following assertions hold:
\begin{itemize}
\item $\mathrm{pcl}(\Gamma_{m,n}^0)$ is an
absolutely irreducible ideal--theoretic complete intersection of
dimension $d+1$ and degree $(d!)^2/(d-m)!(d-n)!$.
\item $\mathrm{pcl}(\Gamma_{m,n}^0)$ is
regular in codimension 2, namely the singular locus of
$\mathrm{pcl}(\Gamma_{m,n}^0)$ has codimension at least 3 in
$\mathrm{pcl}(\Gamma_{m,n}^0)$.
\item $\mathrm{pcl}(\Gamma_{m,n}^0)^\infty$ is a linear
$\fq$--variety of dimension $d$.
\end{itemize}
\end{theorem}

In order to estimate the number of $q$--rational points of
$\Gamma_{m,n}^0$ we shall use a further estimate on the number of
$q$--rational points of a projective complete intersection of
\cite{CaMaPr13}. More precisely, if $V\subset \Pp^N$ is a complete
intersection defined over $\fq$ of dimension $r\geq 2$, degree
$\delta$ and multidegree $\boldsymbol{d}:=(d_1,\ldots,d_{N-r})$,
which is regular in codimension 2, then the following estimate holds
(see \cite[Theorem 1.3]{CaMaPr13}):
\begin{equation}\label{eq: estimate codimension 3 var CaMaPr}
\big||V(\fq)|-p_r\big| \leq 14D^3\delta^2 q^{r-1},
\end{equation}
where $D:=\sum_{i=1}^{N-r}(d_i-1)$.

According to Theorem \ref{th: geometry of Gamma mn s=0}, the
projective variety $\mathrm{pcl}(\Gamma_{m,n}^0)$ satisfies the
hypothesis of \cite[Theorem 1.3]{CaMaPr13}. Therefore, applying
(\ref{eq: estimate codimension 3 var CaMaPr}) we obtain:
$$
\big||\mathrm{pcl}(\Gamma_{m,n}^0)(\fq)|-p_{d+1}\big| \leq
14D_{m,n}^3\delta_{m,n}^2q^d,
$$
where $D_{m,n}:=(m+n)d-\big(m(m+1)+n(n+1)\big)/{2}$ and
$\delta_{m,n}:=(d!)^2/(d-m)!(d-n)!$. Since
$\mathrm{pcl}(\Gamma_{m,n}^0)^{\infty}$ is a linear $\fq$--variety
of dimension $d$,  we have:
\begin{equation}\label{eq: estimate q points pcl Gamma mn s=0}
\big||\Gamma_{m,n}^0(\fq)|-q^{d+1}\big|=
\big||\mathrm{pcl}(\Gamma_{m,n}^0)|-|\mathrm{pcl}(\Gamma_{m,n}^0)^{\infty}|-p_{d+1}+p_d\big|
\leq 14D_{m,n}^3\delta_{m,n}^2q^d.
\end{equation}
Arguing as in Section \ref{subsec: estimate S mn}, we obtain the
following upper bound:
\begin{equation}\label{eq: estimate q points pcl Gamma mn = s=0}
\bigg|\Gamma_{m,n}^0(\fq)\cap
\Big(\left\{B_{0,1}=B_{0,2}\right\}\hskip-0.2cm\bigcup_{1 \leq i <j
\leq m}\hskip-0.2cm \left\{T_i=T_j\right\}\hskip-0.2cm \bigcup_{1
\leq k <l \leq n}\hskip-0.2cm\left\{U_k=U_l\right\}\Big)\bigg| \leq
\xi_{m,n}\delta_{m,n}q^d,\end{equation}
where $\xi_{m,n} :=\binom{m}{2}+\binom {n}{2}+1$. As a consequence,
combining (\ref{eq: relation gamma mn y S mn s=0}), (\ref{eq:
estimate q points pcl Gamma mn s=0}) and (\ref{eq: estimate q points
pcl Gamma mn = s=0}) we deduce the following result.
\begin{theorem}\label{th: estimate S mn s=0}
With assumptions as in Theorem \ref{th: geometry of Gamma mn s=0},
for each $(m,n)$ with $1\leq m,n \leq d$, and $d+1\leq m+n\leq 2d$,
we have
$$
\bigg|\mathcal{S}_{m,n}- \frac{q^{d+1}}{m!n!}\bigg|\leq
\frac{q^d}{m!n!}(14D_{m,n}^3\delta_{m,n}^2+\xi_{m,n}\delta_{m,n}),
$$
where $\xi_{m,n} :=\binom{m}{2}+\binom {n}{2}+1$,
$D_{m,n}:=(m+n)d-\binom{m+1}{2}-\binom{n+1}{2}$ and
$\delta_{m,n}:=\frac{(d!)^2}{(d-m)!(d-n)!}$.
\end{theorem}

Now we proceed as in Section \ref{subsec: asymptotic behavior 2nd
moment}. By Theorem \ref{th: combinatorial reduction variance s=0},
we have
\begin{eqnarray}\mathcal{V}_2(d,0)-\mu_d^2q^2\!\!\!\!&=\!\!\!\!&
  \mathcal{V}(d,0)+
  \displaystyle\mathop{\sum_{1\leq m,n\leq d}}_{2\leq m+n\leq
  d}
  (-q)^{2-n-m}\bigg(\binom{q}{m}\binom{q}{n}-\frac{q^{m+n}}{m!n!}\bigg)
  \nonumber
\\&&\hskip1.25cm +\dfrac{1}{q^{d-1}}\mathop{\sum_{1\leq m,n\leq
d}}_{d+1\leq m+n\leq 2d}
   (-1)^{m+n}\Big(\mathcal{S}_{m,n}-\frac{q^{m+n}}{m!n!}\bigg).\label{eq: expression for V_2(d,0)}
\end{eqnarray}
In Section \ref{subsec: asymptotic behavior 2nd moment} we obtain
the following upper bound for the absolute value $A_1(d,0)$ of the
second term in the right-hand side of (\ref{eq: expression for
V_2(d,0)}):
\begin{equation}\label{eq: cota para A_1(d,0)}
A_1(d,0) \leq 2 \,e \,d \,q.
\end{equation}
On the other hand, in order to bound the last term of (\ref{eq:
expression for V_2(d,0)}), by Theorem \ref{th: estimate S mn s=0} we
have
$$B_1(d,0):=\dfrac{1}{q^{d-1}}\hskip-0.45cm\mathop{\sum_{1\leq m,n\leq
d}}_{d+1\leq m+n\leq 2d}\hskip-0.35cm
\bigg|\mathcal{S}_{m,n}-\frac{q^{d+1}}{m!n!}\bigg|\le
\hskip-0.45cm\mathop{\sum_{1\leq m,n\leq d}}_{d+1\leq m+n\leq
2d}\hskip-0.35cm\bigg(\frac{14D_{m,n}^3\delta_{m,n}^2}{m!n!}+\frac{\delta_{m,n}
\xi_{m,n} }{m!n!}\bigg)q.
$$
With a similar argument as in the proof of Corollary \ref{coro:
estimate V2(d,s,a) without asympt behavior} we see that
\begin{equation}\label{eq: bound B_1(d,0)} B_1(d,0)\leq  \bigg( d^2 2^{2d-2} +14 d^6
\bigg(\sum_{k=0}^{d-1}\binom{d}{k}^{2}(d-k)!\bigg)^2\bigg)q.
\end{equation}
Finally, combining (\ref{eq: V(d,0) Cohen}), (\ref{eq: expression
for V_2(d,0)}), (\ref{eq: cota para A_1(d,0)}) and (\ref{eq: bound
B_1(d,0)}) with the arguments of the proof of Theorem \ref{th: final
estimate V_2(d,s,a)} we deduce the main result of this section.
\begin{theorem}\label{th: average V_2(d,0)}
Assume that $q>d$, $d\ge 5$ for $p>3$ and $d\ge 9$ for $p=3$. Then
we have
$$\left|\mathcal{V}_2(d,0)-\mu_d^2\, q^{2}\right|\le
\big(2^{2d-2}d^2+14^3 d^{2d+8}e^{4\sqrt{d}-2d}\big)q.$$
\end{theorem}
%
%
\appendix
\section{Irreducibility of the discriminant of small families of polynomials}

Let $\K$ be a field and let $\K[X_1,\ldots,X_n]$ be the ring of
multivariate polynomials with coefficients in $\K$. For given
positive integers $a_1,\ldots,a_n$, we define the {\sf weight}
$\mathrm{wt}(\bfs{X}^{\bfs{\alpha}})$ of a monomial
$\bfs{X}^{\bfs{\alpha}}:=X_1^{\alpha_1}\cdots X_n^{\alpha_n}$ as
$\mathrm{wt}(\bfs{X}^{\bfs{\alpha}}):=\sum_{i=1}^{n}a_i\cdot
\alpha_i$. The weight $\mathrm{wt}(f)$ of an arbitrary element
$f\in\K[X_1\klk X_n]$ is the highest weight of all the monomials
arising with nonzero coefficients in the dense representation of
$f$.

An element $f\in\K[X_1,\ldots,X_n]$ is said to be {\sf weighted
homogeneous} (with respect to the weight $\mathrm{wt}$ defined
above) if all its terms have the same weight. Equivalently, $f$ is
weighted homogeneous if and only if $f(X_1^{a_1},\ldots,X_n^{a_n})$
is homogeneous of degree $\mathrm{wt}(f)$. Any polynomial
$f\in\K[X_1,\ldots,X_n]$ can be uniquely written as a sum of
weighted homogeneous polynomials $f=\sum_{i}f_i$, where each $f_i$
is weighted homogeneous with $\mathrm{wt}(f_i)=i$. The polynomials
$f_i$ are called the {\sf weighted homogeneous components} of $f$.
In what follows, we shall use the following elementary property of
weights.
\begin{fact}[{\cite[Proposition 3.3.7]{HeHi11}}]
\label{fact: f_w(f) irred then f irred}
  Let $f\in\K[X_1,\ldots,X_n]$ be a nonconstant polynomial. If the
  component $f_{\mathrm{wt}(f)}$ of highest weight of $f$ is irreducible
  in $\K[X_1,\ldots,X_n]$, then $f$ is irreducible in $\K[X_1,\ldots,X_n]$.
\end{fact}

We shall also use the following simple criterion of irreducibility.
\begin{fact}
\label{fact: criterion irred}
  Let $f\in\K[X_1,\ldots,X_n]$ be a nonconstant polynomial,
  $s<n$,  $R:=\K[X_1,\ldots,X_s]$ and $Q(R):=\K(X_1,\ldots,X_s)$.
  If $f$ is a primitive polynomial of $R[X_{s+1}\klk X_n]$ and an
  irreducible element of $Q(R)[X_{s+1}\klk X_n]$, then $f$ is irreducible
  in $\K[X_1,\ldots,X_n]$.
\end{fact}

Assume that the characteristic $p$ of $\fq$ is not 2. For $d$ and
$s$ with $1\le s \le d-3$, let $B_{d-s-1},\ldots,B_1,B_0,T$ be
indeterminates over $\cfq$ and let
$\bfs{B}_0:=(B_{d-s-1},\ldots,B_1,B_0)$. In what follows, for a
given $\bfs{a}:=(a_{d-1}\klk a_{d-s})\in\cfq^s$, we shall consider
the polynomial $f:=T^d+a_{d-1}T^{d-1}+\cdots+
a_{d-s}T^{d-s}+B_{d-s-1}T^{d-s-1}+\cdots+B_1T+B_0\in\cfq[\bfs{B}_0,T]$.

Denote by $\mathrm{Disc}(f)\in\cfq[\bfs{B}_0]$ the discriminant of
$f$ with respect to the variable $T$. We shall consider the weight
wt of $\cfq[\bfs{B}_0,T]$ defined by setting $\mathrm{wt}(B_j):=d-j$
for $0\le j\le d-s-1$. We observe that, extending this notion of
weight to the polynomial ring $\cfq[B_d\klk B_0]$ in a similar way,
it turns out that the discriminant of a generic degree--$d$
polynomial of $\cfq[B_d\klk B_0][T]$ is weighted homogeneous of
weight $d(d-1)$ (see, e.g., \cite[Lemma 2.2]{FrSm84}).
\begin{theorem}\label{th: irred discrim}
With notations and assumptions as above, $\mathrm{Disc}(f)$ is an
irreducible polynomial of $\cfq[\bfs{B}_0]$.
\end{theorem}
\begin{proof}
First we suppose that $p$ does not divide $d(d-1)$. Consider
$\mathrm{Disc}(f)$ as an element of
$\K_2[B_1,B_0]:=\cfq(B_{d-s-1},\ldots,B_2)[B_1,B_0]$, and consider
the weight $\mathrm{w}_2$ of $\K_2[B_1,B_0]$ defined by setting
$\mathrm{w}_2(B_0):=d$ and $\mathrm{w}_2(B_1):=d-1$. It is easy to
see that the weighted homogeneous component of highest weight of
$\mathrm{Disc}(f)$ is $\Delta_2:=d^d
B_0^{d-1}+(-1)^{d-1}(d-1)^{d-1}B_1^d$. Our assumption of $p$ implies
that $\Delta_2$ is a nonzero polynomial. Furthermore, by the
Stepanov criterion (see, e.g., \cite[Lemma 6.54]{LiNi83}) we deduce
that $\Delta_2$ is irreducible in $\K_2[B_1,B_0]$. Then Fact
\ref{fact: f_w(f) irred then f irred} allows us to conclude that
$\mathrm{Disc}(f)$ is an irreducible element of $\K_2[B_1,B_0]$.
Finally, taking into account that $\mathrm{Disc}(f)$ is a primitive
polynomial of $\cfq[B_{d-s-1},\ldots,B_2][B_1,B_0]$, Fact \ref{fact:
criterion irred} shows that $\mathrm{Disc}(f)$ is irreducible in
$\cfq[\bfs{B}_0]$.

Assume now that $p$ divides $d$. Let
$\K_3:=\cfq(B_{d-s-1},\ldots,B_3)$ and consider $\mathrm{Disc}(f)$
as an element of $\K_3[B_2,B_1,B_0]$. We consider the weight
$\mathrm{w}_3$ of $\K_3[B_2,B_1,B_0]$ defined by setting
$\mathrm{w}_3(B_0)=d$, $\mathrm{w}_3(B_1):=d-1$ and
$\mathrm{w}_3(B_2):=d-2$. If $g:=T^d+B_2T^2+B_1T+B_0$, then $g'=
2B_2 T+B_1$. Therefore, applying the Poisson formula for the
resultant it is easy to prove that
$\mathrm{Disc}(g)=B_1^d+(-1)^{d+1}2^{d-2}B_2^{d-1}B_1^2+(-1)^{d}2^{d}B_2^{d}B_0$.
Since $\deg f=\deg g=d$ and the discriminant of a generic polynomial
of degree $d$ is weighted homogeneous of degree $d(d-1)$, it follows
that $\mathrm{Disc}(g)$ is the component of highest weight of
$\mathrm{Disc}(f)$. Furthermore, we claim that $\mathrm{Disc}(g)$ is
irreducible in $\K_3[B_2,B_1,B_0]$. Indeed, considering
$\mathrm{Disc}(g)$ as a polynomial in $\K_3(B_0)[B_2,B_1]$, we see
that $\mathrm{Disc}(g)$ is the sum of two homogeneous polynomials of
degrees $d$ and $d+1$ without common factors, namely
$B_1^d+(-1)^{d}2^{d}B_2^{d}B_0$ and
$(-1)^{d+1}2^{d-2}B_2^{d-1}B_1^2$ respectively. Then \cite[Lemma
3.15]{Gibson98} proves that $\mathrm{Disc}(g)$ is irreducible in
$\K_3(B_0)[B_2,B_1]$, which in turn implies it is irreducible in
$\K_3[B_2,B_1,B_0]$ by Fact \ref{fact: criterion irred}. Combining
this with Fact \ref{fact: f_w(f) irred then f irred} we deduce that
$\mathrm{Disc}(f)$ is irreducible in $\K_3[B_2,B_1,B_0]$, from which
we readily conclude that it is irreducible in $\cfq[\bfs{B}_0]$ by
Fact \ref{fact: criterion irred}.

Finally, suppose that $p$ divides $d-1$ and consider
$\mathrm{Disc}(f)$ as an element of $\K_3[B_2,B_1,B_0]$. Arguing as
before we conclude that the discriminant $\mathrm{Disc}(g)$ of the
polynomial $g:=T^d+B_2T^2+B_1T+B_0$ is the component of highest
weight of $\mathrm{Disc}(f)$. Observe that $g'=T^{d-1}+2B_2T+B_1$,
and thus
\begin{eqnarray*}
\mathrm{Disc}(g)=
\frac{\mathrm{Res}_T(g,Tg'\!-g)}{\mathrm{Res}_T(g,T)}&=&
\frac{\mathrm{Res}_T(g,B_2T^2\!-B_0)}{\mathrm{Res}_T(g,T)}\\&=&
\frac{\mathrm{Res}_T(T^d+B_1T+2B_0,B_2T^2\!-B_0)}{\mathrm{Res}_T(g,T)}.
\end{eqnarray*}
Applying the Poisson formula for the resultant, we easily deduce the
following identity:
$$\mathrm{Disc}(g)=\left\{\begin{array}{rl}
4B_2^{d}B_0+B_0^{d-1}+4B_0^{d/2}B_2^{d/2}-B_1^2B_2^{d-1}&
\textit{for}\ d\ \textit{even},\\
-4B_2^{d}B_0+B_0^{d-1}+2B_0^{\frac{d-1}{2}}B_2^{\frac{d-1}{2}}-B_1^2B_2^{d-1}
& \textit{for}\ d\ \textit{odd}.\end{array}\right.$$
Then $\mathrm{Disc}(g)$ is irreducible in $\cfq[B_0,B_2][B_1]$ by
the Eisenstein criterion and $\mathrm{Disc}(f)$ is irreducible in
$\K_3[B_2,B_1,B_0]$ by Fact \ref{fact: f_w(f) irred then f irred}.
Arguing as above we obtain that $\mathrm{Disc}(f)$ is irreducible in
$\cfq[\bfs{B}_0]$, finishing thus the proof of the theorem.
\end{proof}


\begin{thebibliography}{CMPP13}

\bibitem[CGH91]{CaGaHe91}
L.~Caniglia, A.~Galligo, and J.~Heintz.
\newblock Equations for the projective closure and effective {Nullstellensatz}.
\newblock {\em Discrete Appl. Math.}, 33:11--23, 1991.

\bibitem[CLO92]{CoLiOS92}
D.~Cox, J.~Little, and D.~O'Shea.
\newblock {\em Ideals, Varieties, and Algorithms: an introduction to
  computational algebraic geometry and commutative algebra}.
\newblock Undergrad. Texts Math. Springer, New York, 1992.

\bibitem[CM06]{CaMa06}
A.~Cafure and G.~Matera.
\newblock Improved explicit estimates on the number of solutions of equations
  over a finite field.
\newblock {\em Finite Fields Appl.}, 12(2):155--185, 2006.

\bibitem[CM07]{CaMa07}
A.~Cafure and G.~Matera.
\newblock An effective {Bertini} theorem and the number of rational points of a
  normal complete intersection over a finite field.
\newblock {\em Acta Arith.}, 130(1):19--35, 2007.

\bibitem[CMP12]{CaMaPr13}
A.~Cafure, G.~Matera, and M.~Privitelli.
\newblock Polar varieties, {Bertini's} theorems and number of points of
  singular complete intersections over a finite field.
\newblock Preprint {\tt arXiv:1209.4938 [math.AG]}, 2012.

\bibitem[CMPP13]{CeMaPePr13}
E.~Cesaratto, G.~Matera, M.~{P\'erez}, and M.~Privitelli.
\newblock On the value set of small families of polynomials over a finite
  field, {I}.
\newblock Preprint {\tt arXiv:1306.1744 [math.NT]}, 2013.

\bibitem[Coh72]{Cohen72}
S.~Cohen.
\newblock Uniform distribution of polynomials over finite fields.
\newblock {\em J. Lond. Math. Soc. (2)}, 6(1):93--102, 1972.

\bibitem[Coh73]{Cohen73}
S.~Cohen.
\newblock The values of a polynomial over a finite field.
\newblock {\em Glasg. Math. J.}, 14(2):205--208, 1973.

\bibitem[Eis95]{Eisenbud95}
D.~Eisenbud.
\newblock {\em Commutative Algebra with a View Toward Algebraic Geometry},
  volume 150 of {\em Grad. Texts in Math.}
\newblock Springer, New York, 1995.

\bibitem[FS84]{FrSm84}
M.~Fried and J.~Smith.
\newblock Irreducible discriminant components of coefficient spaces.
\newblock {\em Acta Arith.}, 44(1):59--72, 1984.

\bibitem[FS08]{FlSe08}
P.~Flajolet and R.~Sedgewick.
\newblock {\em Analytic combinatorics}.
\newblock Cambridge Univ. Press, Cambridge, 2008.

\bibitem[Ful84]{Fulton84}
W.~Fulton.
\newblock {\em Intersection Theory}.
\newblock Springer, Berlin Heidelberg New York, 1984.

\bibitem[Gib98]{Gibson98}
C.~Gibson.
\newblock {\em Elementary geometry of algebraic curves: an undergraduate
  introduction}.
\newblock Cambridge Univ. Press, Cambridge, 1998.

\bibitem[GL02]{GhLa02a}
S.~Ghorpade and G.~Lachaud.
\newblock {\'Etale} cohomology, {Lefschetz} theorems and number of points of
  singular varieties over finite fields.
\newblock {\em Mosc. Math. J.}, 2(3):589--631, 2002.

\bibitem[Har92]{Harris92}
J.~Harris.
\newblock {\em Algebraic Geometry: a first course}, volume 133 of {\em Grad.
  Texts in Math.}
\newblock Springer, New York Berlin Heidelberg, 1992.

\bibitem[Hei83]{Heintz83}
J.~Heintz.
\newblock {Definability} and fast quantifier elimination in algebraically
  closed fields.
\newblock {\em Theoret. Comput. Sci.}, 24(3):239--277, 1983.

\bibitem[HH11]{HeHi11}
J.~Herzog and T.~Hibi.
\newblock {\em Monomial ideals}, volume 260 of {\em Grad. Texts in Math.}
\newblock Springer, London, 2011.

\bibitem[KK90]{KnKn90b}
A.~Knopfmacher and J.~Knopfmacher.
\newblock The distribution of values of polynomials over a finite field.
\newblock {\em Linear Algebra Appl.}, 134:145--151, 1990.

\bibitem[Kun85]{Kunz85}
E.~Kunz.
\newblock {\em Introduction to Commutative Algebra and Algebraic Geometry}.
\newblock Birkh{\"a}user, Boston, 1985.

\bibitem[LN83]{LiNi83}
R.~Lidl and H.~Niederreiter.
\newblock {\em Finite fields}.
\newblock Addison--Wesley, Reading, Massachusetts, 1983.

\bibitem[Sha94]{Shafarevich94}
I.R. Shafarevich.
\newblock {\em Basic Algebraic Geometry: {Varieties} in Projective Space}.
\newblock Springer, Berlin Heidelberg New York, 1994.

\bibitem[Uch55a]{Uchiyama55a}
S.~Uchiyama.
\newblock Note on the mean value of {$V(f)$}.
\newblock {\em Proc. Japan Acad.}, 31(4):199--201, 1955.

\bibitem[Uch55b]{Uchiyama55b}
S.~Uchiyama.
\newblock Note on the mean value of {$V(f)$}. {II}.
\newblock {\em Proc. Japan Acad.}, 31(6):321--323, 1955.

\bibitem[Uch56]{Uchiyama56}
S.~Uchiyama.
\newblock Note on the mean value of {$V(f)$}. {III}.
\newblock {\em Proc. Japan Acad.}, 32(2):97--98, 1956.

\bibitem[Vog84]{Vogel84}
W.~Vogel.
\newblock {\em Results on {B\'ezout}'s theorem}, volume~74 of {\em Tata Inst.
  Fundam. Res. Lect. Math.}
\newblock Tata Inst. Fund. Res., Bombay, 1984.

\end{thebibliography}

\end{document}